\documentclass[12pt]{article}
\usepackage[utf8]{inputenc}
\usepackage[margin=1in]{geometry}
\usepackage{amssymb}
\usepackage{amsmath}
\usepackage{amsthm}
\usepackage{amsfonts}
\usepackage{graphicx}
\usepackage{bbm}
\usepackage[numbers]{natbib}
\usepackage{subcaption}
\usepackage{authblk}

\usepackage{xcolor}
\usepackage{xspace}
\newcommand{\edit}[1]{{\color{black}#1\xspace}}

\newcommand{\Tr}{\text{Tr}}
\renewcommand{\Re}{\mathbb{R}}

\renewcommand{\P}{\mathbb{P}}
\newcommand{\N}{\mathbb{N}_{\mathbf{0}}}
\newcommand{\E}{\mathbb{E}}
\newcommand{\R}{\mathbb{R}}
\newcommand{\Bin}{\mathrm{Bin}}
\renewcommand{\S}{\mathbb{S}}

\newcommand{\calF}{\mathcal{F}}

\newcommand{\Trunc}{\mathtt{Trunc}}

\newcommand{\calH}{\mathcal{H}}
\newcommand{\calG}{\mathcal{G}}

\newcommand{\calB}{\mathcal{B}}

\newcommand{\Var}{\mathrm{Var}}

\newcommand{\bd}{line-crossing}

\newcommand{\wh}[1]{\widehat{#1}}
\newcommand{\wt}[1]{\widetilde{#1}}

\newcommand{\B}{\mathbb{B}}
\newcommand{\err}{\mathrm{err}}

\newcommand{\zhat}{\widehat{Z}}
\newcommand{\estimator}{\widehat{\xi}}
\newcommand{\est}{\widehat{\mu}}
\newcommand{\const}{\mathfrak{C}}

\newtheorem{theorem}{Theorem}
\numberwithin{theorem}{section}
\newtheorem{lemma}[theorem]{Lemma}

\newtheorem{prop}[theorem]{Proposition}
\newtheorem{corollary}[theorem]{Corollary}

\newtheorem{ass}{Assumption}
\newtheorem{remark}{Remark}

\theoremstyle{definition}
\newtheorem{definition}{Definition}

\usepackage{hyperref}
\hypersetup{
	colorlinks=false,
	bookmarks=true,
	breaklinks=true,
	hidelinks=true,
	pdfpagemode=empty,
}

\title{Nonasymptotic Heavy-Tailed Mean Estimation\\ in Smooth Banach Spaces\thanks{Published in \emph{Stochastic Processes and their Applications}.}}

\author[1]{Justin Whitehouse}
\author[2,3]{Ben Chugg}
\author[3]{\\ Diego Martinez-Taboada}
\author[2,3]{Aaditya Ramdas}
\affil[1]{Management Science and Engineering, Stanford University}
\affil[2]{Machine Learning Department, Carnegie Mellon University}
\affil[3]{Department of Statistics and Data Science, Carnegie Mellon University\thanks{\texttt{jwhiteho@stanford.edu, diegomar@andrew.cmu.edu, \{benchugg,aramdas\}@cmu.edu}}
}
\date{}

\begin{document}
\maketitle
\begin{abstract}

\edit{We consider the problem of nonasymptotic mean estimation for heavy-tailed random variables taking values in a smooth Banach space.
In particular, we  extend and improve a simple truncation-based mean estimator by Catoni and Giulini. 
We show that this estimator enjoys finite-sample, high probability deviation bounds that hold uniformly over the class of distributions with a $p$th central moment bounded above by a known constant $v > 0$ for $p \in (1, 2]$.}
Our main contributions exploit connections between truncation-based mean estimation and the concentration of martingales with bounded increments in smooth Banach spaces. We prove two types of time-uniform bounds on the distance between the estimator and unknown mean:  \textit{line-crossing inequalities}, which can be optimized for a fixed sample size $n$, and \textit{iterated logarithm inequalities}, which match the tightness of line-crossing inequalities at all points in time up to a doubly logarithmic factor in $n$. 
Our results do not depend on the dimension of the Banach space, hold under martingale dependence, and all employed constants  are known and small.  

\end{abstract}

\section{Introduction}
Mean estimation is one of the most important primitives in the statistician's toolkit. 
\edit{The goal of nonasymptotic mean estimation is to provide finite-sample, high-probability bounds on the deviations between an estimator and the true, unknown mean. It is generally desired that these can be easily inverted to form confidence intervals for the mean that are valid uniformly over the set of all distributions satisfying certain tail conditions with known constants.}
When the data are light-tailed (perhaps sub-Gaussian, sub-Exponential, or sub-Gamma), the sample mean is the natural estimator of this unknown population mean~\citep{boucheron2003concentration}. However, when the data fail to have finite moments, the naive plug-in mean estimate is known to yield sub-optimal finite sample dependence on the failure probability~\citep{catoni2012challenging}; \edit{see Appendix~\ref{sec:mean_failure}. }

The failure of the plug-in mean has led to a rich literature focused on \emph{heavy-tailed} mean estimation. 
\edit{Estimators such as the thresholded/truncated mean~\citep{tukey1963less, huber1981robust, catoni2018dimension}, trimmed mean~\cite{oliveira2019sub, lugosi2021robust}, median-of-means~\citep{nemirovskij1983problem, jerrum1986random, alon1996space, lugosi2019mean, minsker2015geometric}, and Catoni's M-estimator~\citep{catoni2017dimension, wang2023catoni} have been studied in both univariate and multivariate settings. When a bound on the variance of the observations is known, many of the resulting confidence interval widths enjoy sub-Gaussian or near sub-Gaussian rates~\citep{lugosi2019mean}. In scalar settings, these bounds often gracefully decay when only a bound on the $p$th central moment is known for some $p > 1$~\citep{bubeck2013bandits}. More details are in Section~\ref{sec:related-work}.}

Of the aforementioned statistics, the truncated mean estimator is by far the simplest. This estimator, which involves truncating observations to lie within an appropriately-chosen ball centered at the origin, is extremely computationally efficient and can be updated online, which is very desirable for applied statistical tasks.
However, this estimator also possesses a number of undesirable properties. First, it is not translation invariant, with bounds that depend on the \textit{raw} moments of the random variables. Second, it requires a known bound on the $p$th moment of observations for some $p \geq 2$, thus requiring that the observations have finite variance. Third, bounds are only known in the setting of finite-dimensional Euclidean spaces --- convergence is not understood in the setting of infinite-dimensional Hilbert spaces or Banach spaces. 

The question we consider here is simple: are the aforementioned deficiencies fundamental to truncation-based estimators, or can they be resolved with an improved analysis? The goal of this work is to show that the latter is true, demonstrating how a truncation-based estimator can be improved to handle fewer than two central moments in general classes of Banach spaces.

\subsection{Our Contributions}
\label{sec:contributions}

In this work, we revisit and extend a simple truncation-based mean estimator due to \citet{catoni2018dimension}. Our estimator works by using a small number of samples to produce a pilot mean estimate, say through a sample mean.  Then, the remaining sequence of observations is truncated to lie in an appropriately-sized ball centered at the pilot estimate and then averaged. Centering the remaining observations enables bounds that depend on the central instead of raw $p$-th moment. 
Our estimator improves over that in \citep{catoni2018dimension} in the following ways:
\begin{enumerate}
	\item The analysis holds in arbitrary 2-smooth Banach spaces instead of just finite-dimensional Euclidean space. This includes Hilbert spaces, $L^\alpha$ function spaces, and $\ell^\alpha$ sequence spaces for $2 \leq  \alpha < \infty$.
	\item Our results require only a known upper bound on the conditional central $p$th moment of observations for some $p > 1$. Thus, they are applicable to data lacking finite variance. 
    \item Our bounds hold for data with a martingale dependence structure, not just for independent and identically distributed data.
    \item Our bounds are time-uniform. We prove two types of inequalities: \emph{\bd{} inequalities}, which can be optimized for a target sample size, and \textit{law of the iterated logarithm (LIL) inequalities}, which match the tightness of the \bd{} inequalities at all times simultaneously up to a doubly logarithmic factor in sample size. 
\end{enumerate}
\edit{Our analysis exploits novel connections between heavy-tailed mean estimation and the concentration of bounded random variables in Banach spaces~\citep{pinelis1992approach, pinelis1994optimum}. In particular, it differs significantly from the PAC-Bayesian arguments made in existing analyses of truncation estimators~\citep{catoni2018dimension, chugg2023time, langford2002pac}.}

Informally, if we assume that the central $p$th moments of all observations are conditionally bounded by $v$, and we let $\wh{\mu}_n$ denote our estimate after $n$ samples, then we show that 
\[
\P\left(\|\wh{\mu}_n - \mu\| \geq  \beta C_p v^{1/p}(\log(1/\delta)/n)^{\frac{p-1}{p}} + \err(n, \delta, v)\right) \leq \delta,
\]
where $\beta$ governs the smoothness of the Banach space, $C_p$ is a small, known constant, and $\err(n, \delta, v)$ is a closed-form term satisfying $\lim_{n \rightarrow \infty} \frac{\err(n, \delta, v)}{n^{(p - 1)/p}} = 0$.
As far as we are aware, the only other estimator to obtain a similar guarantee in our setting is Minsker's geometric median-of-means~\citep{minsker2015geometric}; see Appendix~\ref{sec:gmom-banach}. Minsker also works in a Banach space, but assumes that it is separable and reflexive, whereas we will assume that it is separable and smooth; the latter appears stronger.
While we obtain similar rates, our estimator is computationally lightweight and handles martingale dependence (thus allowing our time-uniform guarantees). Finally, our analysis is significantly different from Minsker's and may be of independent interest. 

\subsection{Related Work on Heavy-Tailed Mean Estimation}
\label{sec:related-work}

We now discuss work closely related to heavy-tailed mean estimation. We provide a broader overview in Appendix~\ref{sec:ext-related-work}. Truncation-based estimators originated in the robust statistics literature, dating back to works by Tukey, Huber, and others~\citep{huber1981robust,tukey1963less}. 
These estimators have either been applied in the univariate setting or in $\Re^d$ as in \citet{catoni2018dimension}. 
A related estimator is the trimmed-mean estimator, which removes extreme observations and takes the mean of the remaining points~\citep{oliveira2019sub,lugosi2021robust}. In the setting of real-valued observations, this estimator obtains sub-Gaussian rates~\citep{oliveira2019sub}. 

Separately, \citet{catoni2017dimension} introduced an approach for mean estimation in $\Re^d$ based on M-estimation.
It requires at least two moments, is computationally inefficient, and fails to obtain sub-Gaussian rates.  Follow-up works have improved this estimator in various ways: for real-valued observations, \citet{chen2021generalized} extend it to handle a $p$-th moment for $p\in(1,2)$, \citet{gupta2024beyond} 
refine the constants in the preceding paper, and \citet{mathieu2022concentration} studies the optimality of general M-estimators for mean estimation. 

Median-of-means estimators estimate a heavy-tailed mean by breaking a dataset into several folds, compute a mean estimate on each fold, and then compute some measure of central tendency amongst these estimates~\citep{nemirovskij1983problem,jerrum1986random,alon1996space}. 
For real-valued observations, \citet{bubeck2013bandits} study a median-of-means estimator that holds under infinite variance. Their estimator obtains the same rate as ours and Minsker's.
Most relevant for our work is the result on \textit{geometric median-of-means} due to \citet{minsker2015geometric}, which can be used to aggregate several independent mean estimates in general separable Banach spaces. 
Also related is the multivariate tournament median-of-means estimator due to ~\citet{lugosi2019sub}, which obtains the optimal finite-sample sub-Gaussian rate of order $\sqrt{\Tr(\Sigma)/n} + \sqrt{\|\Sigma\|_{\text{op}}\log(1/\delta)/n}$ but requires i.i.d.\ observations with a shared covariance matrix $\Sigma$.
Median-of-means style approaches have also been extended to general metric spaces~\citep{hsu2016loss,cholaquidis2024gros}, and the tournament median-of-means estimator has been extended to general norms on $\R^d$~\citep{lugosi2019near}. 

\subsection{Preliminaries}
\label{sec:prelims}
We are interested in estimating the shared, conditional mean $\mu$ of a sequence of random variables $(X_n)_{n \geq 1}$ taking values in a Banach space $(\B, \|\cdot\|)$. 
\edit{Recall that a separable Banach space is a complete normed vector space admitting a countable, dense subset. We say $(\B, \|\cdot\|)$ is $(2, \beta)$-smooth if, for all $x, y \in \B$, we have
\begin{equation}
    \label{eq:2beta-smooth}
    \|x+y\|^2 + \|x - y\|^2 \leq 2\|x\|^2 + 2\beta^2 \|y\|^2.
\end{equation}
Note that $\beta \geq 1$ (by taking $x=0$).
\begin{ass}
\label{ass:smooth}
We assume that the Banach space $(\B,\|\cdot\|)$ is separable and $(2, \beta)$-smooth. 
\end{ass}
In the parlance of convex analysis, $(2, \beta)$-smoothness is equivalent to $\|\cdot\|^2$ being $2\beta^2$-smooth, i.e. $
\|tx + (1 - t)y\|^2  \geq t\|x\|^2 + (1 - t)\|y\|^2 - \beta^2t(1 - t)\|x - y\|^2.
$
Assumption~\ref{ass:smooth} is common when studying concentration of measure in Banach spaces~\citep{pinelis1992approach,pinelis1994optimum,howard2020time,martinez2024empirical}. Examples of spaces satisfying Assumption~\ref{ass:smooth} include separable Hilbert spaces ($\beta = 1$), $\ell^\alpha$ sequence spaces/$L^\alpha$ spaces of functions for $\alpha \in [2, \infty)$ (which have $\beta = \sqrt{\alpha - 1}$).} We make the following assumption about $(X_n)_{n \geq 1}$.

\begin{ass}
\label{ass:mean}
Let $(X_n)_{n \geq 1}$ be a discrete time sequence of $\B$-valued random variables \edit{defined on a probability space $(\B, \calB, \P)$ where $\calB$ is the corresponding Borel $\sigma$-algebra}. We assume $(X_n)_{n \geq 1}$ are adapted to a filtration $\calF \equiv  (\calF_n)_{n\geq 0}$ such that 
\begin{enumerate}
    \item[(1)] $\E(X_n \mid \calF_{n - 1}) = \mu$ almost surely, for all $n \geq 1$ and some (unknown) $\mu \in \B$, and 
    \item[(2)] $\sup_{n \geq 1}\E\left(\|X_n - \mu\|^p \mid \calF_{n - 1}\right) \leq v < \infty$ almost surely for some (known) constants $p \in (1, 2]$ and $v > 0$.
\end{enumerate}
\end{ass}
\edit{Since $(2, \beta)$-smooth spaces are reflexive \citep{nowak2013group}, they possess the Radon-Nikodym property~\citep{diestel1976radon}, so the conditional expectations in  Assumption~\ref{ass:mean} are well-defined.} We note that point (2) implies that $\E\|X_n\| < \infty$ for each $n \geq 1$, which in turn implies that expectations $\E X_n$ exist as defined via Bochner integrals.
The martingale dependence in condition (1) above is weaker than the traditional i.i.d.\ assumption, requiring only a constant \emph{conditional} mean. 
Meanwhile, condition (2) allows for infinite variance, a weaker moment assumption than past works studying concentration of measure in Banach spaces~\citep{minsker2015geometric,pinelis1992approach,pinelis1994optimum}.

\paragraph{Notation} For notational simplicity, we define the conditional expectation operator $\E_{n - 1}[\cdot]$ to be $\E_{n - 1}[X] := \E(X \mid \calF_{n - 1})$ for any $n \geq 1$. 
If $S \equiv (S_n)_{n \geq 0}$ is some stochastic process, we denote the $n$-th increment as $\Delta S_n := S_n - S_{n - 1}$ for any $n \geq 1$.
For any process or sequence $a \equiv (a_n)_{n \geq 1}$, denote by $a^{(n)}$ the first $n$ values: $a^{(n)} = (a_1,\dots,a_n)$. 
We say the process $S$ is \emph{predictable} with respect to a filtration $\calF$ if $S_n$ is $\calF_{n-1}$-measurable for all $n\geq 1$. 
Our analysis will make use of Gateaux derivatives of functions in a Banach space; see \citet{ledoux2013probability} for a definition. 
\edit{Finally, for functions $f,g:\N \rightarrow \R$, we write  $f\sim g$ if $\lim_{n \rightarrow \infty}\frac{f(n)}{g(n)} = 1$. 
}

\paragraph{Outline}
Section~\ref{sec:results} provides statements of the main results. Theorem~\ref{thm:template} is a general template for obtaining time-uniform bounds on truncation-style estimators. Corollary~\ref{cor:opt-lambda} then optimizes the template for tightness at a fixed sample size. Section~\ref{sec:analysis} is dedicated to the proof of Theorem~\ref{thm:template}. 
Section~\ref{sec:lil} then uses a technique known as ``stitching'' to extend our \bd{} inequalities to bounds which shrink to zero over time at an iterated logarithm rate. Finally, Section~\ref{sec:experiments} provides several numerical experiments demonstrating the efficacy of our proposed estimator.

\section{Main Result}
\label{sec:results}

\edit{We now define our truncation-based estimator. Unlike the original estimator due to \citet{catoni2018dimension}, which projects observations onto a ball centered at the origin, our estimator truncates observations around a pilot mean estimate constructed from a small number of samples. More formally, our estimate of $\mu$ at time $n$ will be 
\begin{equation}
    \label{eq:estimator}
        \est_n(k) \equiv \est_n(k,\lambda, \zhat_k) := \frac{1}{n - k}\sum_{m = k + 1}^n\big\{ \Trunc(\lambda (X_m - \zhat_k))(X_m - \zhat_k) + \zhat_k\big\}, 
\end{equation}
where $\zhat_k$ is a preliminary mean estimate formed using $X_1,\dots,X_k$,  
$\lambda>0$ is a fixed hyperparameter, and $\Trunc : \B \rightarrow [0, 1]$ is defined as $\Trunc(x) := \frac{1 \land \|x\|}{\|x\|}$ and $\Trunc(0) := 0$. In words $\Trunc(\lambda x) x$ is the projection of $x$ onto the ball of radius $\lambda^{-1}$ in $\B$. Likewise, $\Trunc(\lambda(x - y))(x - y)$ is the projection of $x$ onto the ball of radius $\lambda^{-1}$ centered at $y \in \B$. Defining $\zhat_0=0$ when $k=0$, we note that $\est_n(0)$ is the usual, origin-centered truncation estimator, analyzed by \citet{catoni2018dimension} and \citet{chugg2023time}. }

\edit{In order to analyze $\est_n(k)$, we need to assume some rate of convergence for the pilot estimate $\zhat_k$. The following assumption codifies this, with the rate function $r(\delta, k)$ controlling the deviations between $\zhat_k$ and the mean $\mu$ with high probability.
\begin{ass}
\label{ass:pilot}
 We assume that there is a function $r:(0,1]\times \mathbb{N} \to\Re_{\geq 0}$ such that, for all $k\geq 1$, the preliminary mean estimator $\zhat_k$ satisfies:
\begin{align}
    \forall\delta\in(0,1), \ 
    \P(\|\mu - \zhat_k\| \geq r(\delta,k))\leq \delta .
    \label{eq:zhat_uniform_bound}
\end{align}
\end{ass}
}

The following ``template'' bound is our main result,  which we later instantiate with different choices of $(\zhat_k, r, \lambda)$  to yield specific tight bounds.

\begin{theorem}[Main result]
\label{thm:template} 
 Let $(X_n)_{n\geq 1}$ be a sequence of random variables satisfying Assumption~\ref{ass:mean} which lie in some Banach space $(\B,\|\cdot\|)$ satisfying Assumption~\ref{ass:smooth}.
 Suppose we have a preliminary mean estimator $\zhat_k$ that satisfies Assumption~\ref{ass:pilot} for some rate function $r$.
Fix any $k\geq 1$ and $\delta_1,\delta_2 \in (0, 1]$. Then, for any $\lambda>0$, the mean estimator in~\eqref{eq:estimator} satisfies
\begin{multline}
\label{eq:template-guarantee}
\P\left(\forall n > k,\left\| \est_{n}(k) - \mu\right\| \leq \lambda^{p-1}\beta C_p (v + r(\delta_2,k)^p) + \frac{\beta \log(2/\delta_1)}{\lambda (n-k)}\right) \\
\geq 1-\delta_1-\delta_2,
\end{multline}
\edit{where $C_p = \const_p(\B) + \beta^{-1}K_p 2^{p - 1}$, $K_p = \tfrac{1}{p/q + 1}\left(\tfrac{p/q}{p/q + 1}\right)^{p/q}$, $q = (1 - \tfrac{1}{p})^{-1}$,  $\const_p(\B) = 2^{p-1} (\frac{e^2 -3}{4})$ if $\B$ is a Hilbert space, else $\const_p(\B) = 2^{p+1} \left(\frac{e^{2\beta^{-1}} -2\beta^{-1} - 1}{(2\beta^{-1})^2}\right)$.}
\end{theorem}

\edit{We emphasize that all constants in Theorem~\ref{thm:template} are known and small. In particular, one can check that $K_p< 1$ for all $p>1$. Further, for any $\beta \geq 1$, $\frac{e^{2\beta^{-1}} - 2\beta^{-1} - 1}{(2\beta^{-1})^2} \leq \left(\frac{e^2 - 3}{4}\right) < 1.1$, and hence $\const_p(\B) \leq 2^{p + 1} \cdot 1.1$ with a further factor of four being saved if $\B$ is a Hilbert space.}

\edit{We now interpret the role of the rate function $r(\delta, k)$ in Theorem~\ref{thm:template}.} 
$r(\delta,k)$ may only decay to zero slowly as $k \to \infty$, but as the total sample size $n$ grows, the effect of initial mean estimation will be negligible as long as we choose $k=k(n)$ to increase to infinity with $n$ (e.g., $k=\log n$ or $k=\sqrt n$). 
Let us note two special cases of the function $r(\delta, k)$: 
\begin{enumerate}
    \item When we take $\zhat_k := \frac{1}{k}\sum_{i = 1}^k X_i$ as the sample mean, Lemma~\ref{lem:mean_concentration} allows us to take $r(\delta, k) := \frac{2 \beta v^{1/p}}{\delta^{1/p}k^{(p - 1)/p}}$.
    \item When we use the geometric median-of-means estimator due to \citet{minsker2015geometric} and assume all observations $X_1, \dots, X_k$ are i.i.d., our discussion in Appendix~\ref{sec:gmom-banach} allows us to set $r(\delta, k) := \frac{11\beta}{0.1^{1/p}} v^{1/p} \left(\frac{3.5 \log(1/\delta)+1}{k}\right)^{(p - 1)/p}$.
\end{enumerate}
The latter provides a better dependence on $\delta$ but comes at an additional computational cost. In our experiments (Section~\ref{sec:experiments}), we see this benefit is apparent for small sample sizes, but washes out as $n$  grows. Further, the rate function $r(\delta, k)$ for geometric median means requires observations to be i.i.d., something not required by the sample mean.

Theorem~\ref{thm:template} is a \emph{\bd{} inequality} in the spirit of \citet{howard2020time}. That is, if we multiply both sides by $n - k$, it provides a time-uniform guarantee on the probability that the deviation between $\est_n$ and $\mu$ ever crosses the line parameterized by the right hand side of \eqref{eq:template-guarantee}.  If we optimize the value of $\lambda$ for a particular sample size $n$, the bound will remain valid for all sample sizes, but will be tightest for the optimized number of samples. To obtain bounds that are tight for all $n$ simultaneously, one must pay an additional iterated logarithmic price in $n$ (see Section~\ref{sec:lil}). 
We now show how to choose $\lambda$ in Theorem~\ref{thm:template} to optimize the tightness for a fixed sample size. The proof follows from analytically minimizing the right hand side of~\eqref{eq:template-guarantee}. 

\begin{corollary}[\edit{Fixed-$n$ optimization}]
\label{cor:opt-lambda}
\edit{In the setup of Theorem~\ref{thm:template}, we have}
\begin{multline}
\label{eq:cor_bound}
\P\left(\|\wh{\mu}_n(k) - \mu\| \leq 2\beta \Big( C_p(v + r^p( \delta_2,k))\Big)^{1/p}\left(\frac{\log(2/\delta_1)}{n - k}\right)^{(p - 1)/p}\right) \\
\geq 1-\delta_1-\delta_2.
\end{multline} 

\end{corollary}

\edit{Note that unlike Theorem~\ref{thm:template}, the above corollary is not time-uniform in $n$. We are instead first fixing a sample size $n$ and optimizing $\lambda$ accordingly\footnote{Here, the optimizing choice of $\lambda$ is $\lambda = \left(\frac{\beta \log(2/\delta_1)}{( \beta \const_p(\B) + K_p 2^{p - 1})(n-k)(v + r^p( \delta_2,k))}\right)^{1/p}.$}. This optimal choice of $\lambda$ does yield a time-uniform bound in the ``template'' theorem, albeit a loose one for sample sizes $t \gg n$.} 

To better understand the implications of the above corollary, consider a sequence of estimators $\wh{\mu}_n(k)$ indexed by $n$, with $k=k(n) \to \infty$ as $n\to\infty$ such that $k/n\to 0$ and $r(\delta,k)\to 0$ (which is clearly satisfied by the aforementioned examples). Then, if $\delta_1,\delta_2$ are kept constant with $\delta :=\delta_1+\delta_2$, we have the following asymptotic equivalence:
\begin{equation}
\label{eq:asymptotic-rate}
2\beta \Big( C_p(v + r^p( \delta_2,k))\Big)^{\frac{1}{p}}\left(\frac{\log(2/\delta_1)}{n - k}\right)^{\frac{p - 1}{p}} \sim 2\beta C_p^{\frac{1}{p}} v^{\frac{1}{p}}\left(\frac{\log(2/\delta_1)}{n }\right)^{\frac{p -1}{p}},
\end{equation}
which follows because 
\[
(v + r^p( \delta_2,k))^{1/p} \leq v^{1/p}  +  r( \delta_2,k) \qquad (\text{Since }(a + b)^{1/p} \leq a^{1/p} + b^{1/p})
\]
and because $\lim_{n \rightarrow \infty}r(\delta_2, k(n)) = 0$.

With this in mind, we see our rates match the rates of \citet{bubeck2013bandits} in scalar settings and \citet{minsker2015geometric} in Banach spaces. \edit{We note that the sample mean provably has exponentially-worse dependence than our estimator on the failure probability $\delta$ --- see Appendix~\ref{sec:mean_failure} for a discussion.}
The following result, due to \citet{lugosi2019mean} and \citet{devroye2016sub} from the scalar and i.i.d.\ setting, shows the dependence of our bound on $v, n,$ and $\delta$ is generally unimprovable in a mini-max sense.
\begin{prop}
\label{prop:lower_bd_mean}
Let $p \in (1, 2]$, $v > 0$, $\delta \in (0, 1)$, and $n \geq 1$ be arbitrary. Further, let $\wh{\mu}_n$ be an arbitrary mean estimator.
Then, there exists $\mu \in \R$ and i.i.d.\  $X_1, \dots, X_n$ with $\E X_n = \mu$ and $\E|X_n - \mu|^p = v$ such that
\begin{align*}
\P\left(\left|\wh{\mu}_n - \mu\right| > v^{1/p}\left(\frac{\log(2/\delta)}{n}\right)^{(p - 1)/p}\right) \geq \delta.
\end{align*}
\end{prop}

\begin{remark}
\label{rem:big_p}
We note that throughout this section we have explicitly focused on the setting $p \in (1, 2]$. One may ask how the truncation-based estimator can be adapted to handle a moment bound of $\sup_{n}\E_{n - 1}\|X_n - \mu\|^p \leq v$ for some $p > 2$. Jensen's inequality trivially  for any $n \geq 1$
\[
\E_{n - 1}\|X_n - \mu\|^2 \leq \left(\E_{n -1}\|X_n - \mu\|^p\right)^{2/p} \leq v^{2/p},
\]
and thus Corollary~\ref{cor:opt-lambda} can be readily applied.
\end{remark}

\section{Proof of Theorem~\ref{thm:template}}
\label{sec:analysis}
We will prove a slightly more general result that reduces to Theorem~\ref{thm:template} in a special case. 
In this section, fix two arbitrary $\calF$-predictable sequences,  $(\zhat_n)_{n \geq 1}$ taking values in $\B$ and $(\lambda_n)_{n \geq 1}$ taking values in $\R$.  Define 
\begin{align}
Y_m &:= X_m - \zhat_m, \text{ and }\\
    \estimator_n \equiv \estimator_n(\lambda^{(n)},\zhat^{(n)}) &:= \sum_{m\leq n}\lambda_m\big\{ \Trunc(\lambda_m Y_m)Y_m + \zhat_m\big\}.
\end{align}
If we take $\lambda_n$ to be constant $\lambda$ and $\zhat_m = \zhat$ to be $\calF_0$-measurable, then $\estimator_n = n\lambda \est_n$. 
However, our more general analysis allows us to consider sequences of predictable mean estimates, which will be useful in deriving Theorem~\ref{thm:stitching} in the next section.

Our preliminary goal is to find, for any $\rho >0$, $(\zhat_n)_{n \geq 1}$, and $(\lambda_n)_{n \geq 1}$ a process $(V_n)_{n\geq 0} \equiv (V_n(\lambda^{(n)}, \rho, \zhat^{(n)}))_{n \geq 0}$ such that the process $(M_n(\lambda^{(n)}))_{n \geq 0}$ given by
\begin{equation}
\label{eq:Mn}
    M_n(\lambda^{(n)}) = \frac12 \exp\left\{\rho \bigg\| \estimator_n - \mu\sum_{m\leq n}\lambda_m\bigg\| - V_n(\lambda^{(n)}, \rho, \zhat^{(n)})\right\},
\end{equation}
is upper bounded by a nonnegative supermartingale. (In particular, it will be upper bounded by a nonnegative supermartingale with initial value 1, meaning that, in recent parlance, $M_n$ is an E-process~\citep{ramdas2024hypothesis}.) Applying Ville's inequality will then give us a time-uniform bound on the deviation of the process $\| (\sum_{m\leq n}\lambda_m)^{-1}\estimator_n - \mu\| $ in terms of $V_n$. 

To this end, we start by defining, for any $\rho > 0$, the constant
\[
\const_p(\B, \rho) := \begin{cases}
2^{p - 1}\left(\frac{e^{2\rho} - 2\rho - 1}{(2\rho)^2}\right) &\text{if }(\B,\|\cdot\|)\text{ is a Hilbert space}, \\  
         2^{p+1} \left(\frac{e^{2\rho} - 2 \rho - 1}{(2\rho)^2}\right), &\text{otherwise},
\end{cases}
\]
and note that $\const_p(\B) = \const_p(\B, \beta^{-1})$. We additionally define the variance process
\begin{align}
    &V_n \equiv V_n(\lambda^{(n)}, \rho, \zhat^{(n)}) = (\rho^2\beta^2\const_p(\B, \rho) + \rho K_p2^{p-1}) G_n,  \notag \\ 
    \text{where}\quad & 
    G_n \equiv G_n(\lambda^{(n)},\zhat^{(n)}):= \sum_{m\leq n} \lambda_m^p( v + \| \mu - \zhat_m\|^p).\label{eq:Gn}
\end{align}
$(G_n)_{n \geq 0}$ gives a weighted measure of the deviation of the naive estimates $\zhat_1,\dots,\zhat_n$ from $\mu$, and $(V_n)_{n \geq 0}$ further couples this deviation with the smoothness of the Banach space.
Since it is difficult to reason about the difference between $\estimator_n$ and $\mu\sum_{m\leq n}\lambda_m$ directly, we introduce the proxy 
\begin{equation}
\label{eq:mutilde}
    \widetilde{\mu}_n(\lambda_n) := \E_{n-1}[\Trunc(\lambda_n Y_n)Y_n] + \zhat_n,
\end{equation}
so that $\lambda_n \widetilde{\mu}_n(\lambda_n)$ is just the conditional mean of $n$th increment of $(\estimator_n)_{n \geq 0}$ given $\calF_{n-1}$. We then separately argue about distances $\|\estimator_n - \sum_{m\leq n}\lambda_m\widetilde{\mu}_m(\lambda_m)\|$ and $\lambda_m \| \widetilde{\mu}_m(\lambda_m) - \mu\|$. This gives us a bound on $\|\estimator_n - \mu\sum_{m\leq n}\lambda_m\|$ via the triangle inequality.

To build some intuition, if the $\lambda_m$ are all small, we are truncating onto a very large ball, and thus expect $\sum_{m \leq n} \lambda_m \|\wt{\mu}_m(\lambda_m) - \mu\|$ to be small. However, in this case, our ability to bound $\|\estimator_n - \sum_{m \leq n}\lambda_m\wt{\mu}_m(\lambda_m)\|$ using concentration of measure results for bounded random variables will be weak. Likewise, if the $\lambda_m$ are large, $\wt{\mu}_m(\lambda_m)$ will be far from $\mu$, but we will have greater control over the growth of $\estimator_n$. Our argument in the sequel involves carefully balancing the growth of each term to give strong rates of concentration.

\subsection{Step 1: Bounding $\| \widetilde{\mu}_n(\lambda_n) - \mu\|$}

We need the following analytical property of $\Trunc$, which will be useful in bounding the truncation error with fewer than two moments. We note that the following lemma was used by \citet{catoni2018dimension} for $k \geq 1$. We extend their result to hold for  $k > 0$. 

\begin{lemma}
\label{lem:kth}
For any $k > 0$ and $x\in \B$, we have that
\[
1 - \Trunc(x) \leq \frac{\|x\|^k}{k + 1}\left(\frac{k}{k + 1}\right)^k.
\]
\end{lemma}
\begin{proof}
Fix $k > 0$. It suffices to show that $f(t) := 1 - \frac{1 \land t}{t} \leq \frac{t^k}{k + 1}\left(\frac{k}{k + 1}\right)^k =: g_k(t)$ for all $t \geq 0$. For $t \in [0, 1],$ the result is obvious. For $t \geq 1$, we need to do a bit of work. 
First, note that $g_k(1) > f(1) = 0$, and that both $g_k$ and $f$ are continuously differentiable and monotone increasing. Further, we only have $g_k(t) = f(t)$ precisely when $t = \frac{k + 1}{k}$. Let this value of $t$ be $t^\ast$. This immediately implies that $g_k(t) \geq f(t)$ for $t \in [1, t^\ast]$. 
To check the inequality for all $t \geq t^\ast$, it suffices to check that $f'(t) < g_k'(t)$. We verify this by direct computation. First, $f'(t) = \frac{1}{t^2}$. Likewise, we have that $g_k'(t) = t^{k - 1}\left(\frac{k}{k + 1}\right)^{k + 1}$. Taking ratios, we see that
\[
\frac{g_k'(t)}{f'(t)} = t^{k + 1}\left(\frac{k}{k + 1}\right)^{k + 1} \geq \left(\frac{k + 1}{k}\right)^{k + 1}\left(\frac{k}{k + 1}\right)^{k + 1} = 1,
\]
proving the desired result.
\end{proof}

We can now proceed to bounding $\| \widetilde{\mu}_n(\lambda_n) - \mu\|$. 

\begin{lemma}
\label{lem:trunc_error}
Let $(X_n)$ be a sequence of $\B$-valued random variables satisfying Assumption~\ref{ass:mean}.
Let $\zhat_n$ and $\lambda_n$ be $\calF_{n-1}$-predictable and $\widetilde{\mu}_n$ be as in \eqref{eq:mutilde}. 
Then: 
\[
\|\mu - \widetilde{\mu}_n(\lambda_n)\| \leq K_p 2^{p - 1}\lambda_n^{p - 1}(v  + \|\zhat_n - \mu\|^p),
\]
\edit{where $K_p$ was defined in Theorem~\ref{thm:template}}.
\end{lemma}

\begin{proof} Since $\zhat_n$ and $\lambda_n$ are predictable, we may treat them respectively as constants $z$ and $\lambda$ when conditioning on $\calF_{n-1}$. 
Using Holder's inequality, 
1\\
where $1/p + 1/q=1$. 
The second expectation on the right hand side can be bounded using Minkowski's inequality and the fact that $\|\cdot\|^p$ is convex for $p\geq 1$:  
\begin{align}
\E_{n-1}[\|X_n - z\|^p] &= \E_{n-1}[\|X_n - \mu + \mu - z\|^p] \notag \\
&\leq 2^{p - 1}\left(\E_{n-1}[\|X_n - \mu\|^p] + \|z - \mu\|^p\right) \notag \\
&\leq 2^{p - 1}(v + \|z - \mu\|^p). \label{eq:Xn-z_bound}
\end{align}
Next, by Lemma~\ref{lem:kth}, we have for any $k > 0$,
\[
\E_{n-1}[|1 - \Trunc(\lambda (X_n-z))|^q] \leq \E_{n-1}\left[\left(\frac{\lambda^k\|X_n-z\|^k}{k + 1}\left(\frac{k}{k + 1}\right)^k\right)^q\right].
\]
In particular, selecting $k = \frac{p}{q}$, we have
\begin{align*}
\E_{n-1}[|1 - \Trunc(\lambda (X_n-z))|^q] &\leq K_p^q\E_{n-1}\left[\lambda^p\|X_n-z\|^p\right] \\
&\leq K_p^q \lambda^p 2^{p-1}(v + \|z - \mu\|^p).
\end{align*}
 Piecing everything together, we have that
\[
\E_{n-1}[|1 - \Trunc(\lambda(X_n - z))|^q]^{1/q} \leq K_p \lambda^{p/q} 2^{(p - 1)/q} (v + \|z - \mu\|^p)^{1/q}.
\]
Therefore, recalling that $p/q = p - 1$, we have $
\| \mu - \widetilde{\mu}_n(\lambda)\| \leq K_p\lambda^{p-1} 2^{p - 1}(v + \|z - \mu\|^p)$, 
which is the desired result. 
\end{proof}

\subsection{Step 2: Bounding $\|\estimator_n(\lambda^{(n)}, \zhat^{(n)}) - \sum_{m\leq n}\lambda_m \widetilde{\mu}_m(\lambda_m)\|$}
We can now proceed to bounding $\|\estimator_n(\lambda^{(n)}, \zhat^{(n)}) - \sum_{m\leq n}\lambda_m \widetilde{\mu}_m(\lambda_m)\| = \|S_n(\lambda^n, \zhat^n)\|$ where we have defined the process $(S_n)_{n \geq 0}$ by
\begin{equation}
\label{eq:Sn}
    S_n \equiv S_n(\lambda^{(n)}, \zhat^{(n)}) := \sum_{m=1}^n \lambda_m \left\{\Trunc(\lambda_m Y_m) Y_m + \zhat_m - \widetilde{\mu}_m(\lambda_m)\right\},
\end{equation}
and $\widetilde{\mu}_m(\lambda_m)$ is as in~\eqref{eq:mutilde}. 
Note that $(S_n)_{n \geq 0}$ is a martingale with respect to $\calF$. 
The following proposition is the most technical result in the paper.  It follows from a modification of the proof of Theorem 3.2 in \citet{pinelis1994optimum}, combined with a Bennett-type inequality for 2-smooth separable Banach spaces presented in \citet[Theorem 3.4]{pinelis1994optimum}. 
We present the full result here, even those parts found in Pinelis' earlier work, for the sake of completeness. 

\begin{prop}
\label{prop:banach}
Let $(X_n)_{n \geq 1}$ be a process satisfying Assumption~\ref{ass:mean} and lying in a Banach space $(\B,\|\cdot\|)$ satisfying Assumption~\ref{ass:smooth}. Then, the exponential process
\[
U_n \equiv U_n(\lambda^{(n)}, \rho, \zhat^{(n)}) := \frac12\exp\left\{\rho\left\|S_n(\lambda^{(n)}, \zhat^{(n)})\right\| - \rho^2 \beta^2\const_p(\B, \rho)  G_n\right\}
\]
is bounded from above by a nonnegative supermartingale with initial value $1$, where $G_n$ is defined by~\eqref{eq:Gn}. 
\end{prop}

\begin{proof}
Fix some $n \geq 1$ and let $U_n$ be as above. We first observe that
\begin{align*}
\|\Delta S_n\| 
&= \lambda_n\|\Trunc(\lambda_n Y_n) Y_n + \zhat_n- \wt{\mu}_n(\lambda_n)\| \\
&\leq \lambda_n \|\Trunc(\lambda_n Y_n)Y_n\| + \lambda_n \|\zhat_n - \wt{\mu}_n(\lambda_n)\| \leq 2,
\end{align*}
by definition of $\Trunc$. 
Let $T_n = \Trunc(\lambda_nY_n)Y_n$. 

If $(\B,\|\cdot\|)$ is a Hilbert space with inner product $\langle \cdot, \cdot\rangle$ (which induces $\|\cdot\|$), then 
\begin{align*}
\E_{n-1} \|\Delta S_n\|^2  = \lambda_n^2\E_{n-1} \langle T_n - \E_{n-1} T_n, T_n - \E_{n-1} T_n\rangle \leq \lambda_n^2\E_{n-1} \|T_n\|^2. 
\end{align*}
Otherwise, we have
\begin{align*}
\E_{n - 1}\|\Delta S_n\|^2 
&\leq \lambda_n^2\{\E_{n - 1}\left(\|T_n\| + \|\E_{n - 1}T_n\|\right)^2\} & \\
&\leq 2\lambda_n^2\{\E_{n - 1}\|T_n\|^2 + \|\E_{n - 1}T_n\|^2 \}
\leq 4\lambda_n^2\E_{n - 1}\|T_n\|^2,
\end{align*}
where the penultimate inequality uses that $(a + b)^2 \leq 2a^2 + 2b^2$ and the final inequality follows from Jensen's inequality. 
Therefore, we can write 
\begin{equation}
    \E_{n-1} \|\Delta S_n\|^2 \leq C \lambda_n^2 \E_{n-1} \|T_n\|^2, 
\end{equation}
where $C=1$ if $\|\cdot\|$ is induced by an inner product, and $C=4$ otherwise. We note that this extra factor of 4 is responsible for the two cases in the definition of $\const_p(\B, \rho)$ in Theorem~\ref{thm:template}. 
Carrying on with the calculation, write 
\begin{align}
\E_{n-1}\|\Delta S_n\|^2 &\leq  C\lambda_n^2\E_{n - 1}\left[\|T_n\|^p \|T_n\|^{2 - p}\right] \notag \\
&\leq C \lambda_n^{p}\E_{n - 1}\left\|T_n\right\|^p \notag \\
&\leq C \lambda_n^{p }\E_{n - 1}\left\|Y_n\right\|^p \notag \\
&\leq C\lambda_n^{p } 2^{p -1}(v + \|\mu - \zhat_n\|^p),\label{eq:Sn2_bound}
\end{align}
where the final inequality follows by the same argument used to prove~\eqref{eq:Xn-z_bound} in Lemma~\ref{lem:trunc_error}. We have shown that the random variable $\|\Delta S_n\|$ is bounded and its second moment (conditioned on the past) can be controlled, which opens the door to Pinelis-style arguments (see \citet[Theorem 3.4]{pinelis1994optimum} in particular). 
Define the function $\varphi : [0, 1] \rightarrow \R_{\geq 0}$ by
\[
\varphi(\theta) := \E_{n - 1}\cosh\left(\rho\|S_{n - 1} + \theta\Delta S_n\|\right).
\]

In principle, the norm function need not be differentiable, and so the same applies to $\varphi$. However, \citet{pinelis1994optimum} proved that one may assume smoothness of the norm without loss of generality (see \citet[Remark 2.4]{pinelis1994optimum}). Thus, a second order Taylor expansion yields
\begin{align*}
\E_{n - 1}\cosh(\rho\|S_n\|) &= \varphi(1) = \varphi(0) + \varphi'(0) + \int_0^1 (1-\theta)\varphi''(\theta)d\theta.
\end{align*}
Observe that
\begin{align*}
\varphi''(\theta) &\leq \rho^2\beta^2  \E_{n - 1}\left[\|\Delta S_n\|^2 \cosh(\rho\|S_{n - 1}\|)e^{ \theta\rho \|\Delta S_n\|}\right] \\
&\leq \rho^2\beta^2 \cosh(\rho\|S_{n - 1}\|) \E_{n - 1}\left[\|\Delta S_n\|^2 \right] e^{2\rho\theta},
\end{align*}
where the first inequality follows from the proof of Theorem~3.2 in \citet{pinelis1994optimum} and Theorem~3 in \citet{pinelis1992approach}, and the second inequality is obtained in view of $\|\Delta S_n\| \leq 2$.

Next, by the chain rule, we have
\begin{align*}
\varphi'(0) &= \frac{d}{d\theta}\left(\E_{n - 1}\cosh(\rho\|S_{n - 1} + \theta \Delta S_n\|)\right)\big\vert_{\theta = 0} \\
&= \E_{n - 1}\left[\frac{d}{d\theta}\cosh(\rho\|S_{n - 1} + \theta \Delta S_n\|) \Big\vert_{\theta = 0}\right] \\
&= \rho\E_{n - 1}\left[\langle D_f\|f\|\big\vert_{f = S_{n - 1}}, \Delta S_n\rangle  \cdot \frac{d}{dx}\cosh(x)\big\vert_{x = \|S_{n - 1}\|}\right] \\
&= \rho\left\langle D_f\|f\|\big\vert_{f = S_{n - 1}},\E_{n - 1}\Delta S_{n} \right\rangle  \cdot \frac{d}{dx}\cosh(x)\big\vert_{x = \|S_{n - 1}\|} \\
&= 0, 
\end{align*}
where $\left\langle D_f \varphi(f)\mid_{f = g}, y - x\right\rangle$ denotes the Gateaux derivative of $\varphi$ with respect to $f$ at $g$ in the direction of $y - x$.
The final equality follows from the fact that $(S_n)_{n \geq 0}$ is itself a martingale with respect to $(\calF_n)_{n \geq 1}$. Thus, leveraging that $\varphi'(0) = 0$, we have
\begin{align*}
&\E_{n - 1}\cosh(\rho\|S_n\|) = \varphi(0) + \varphi'(0) + \int_0^1 (1-\theta)\varphi''(\theta)d\theta \\
&\qquad \leq \cosh(\rho\|S_{n - 1}\|)  \left( 1 + \rho^2\beta^2 \E_{n - 1}\left[\|\Delta S_n\|^2  \right] \int_0^1  (1-\theta) e^{2 \rho\theta} d\theta \right)\\
&\qquad \stackrel{(i)}{=} \cosh( \rho\|S_{n - 1}\|)  \left( 1 + \rho^2\beta^2  \left(\frac{e^{2\rho} - \rho - 1}{(2\rho)^2} \right)  \E_{n - 1}\left[\|\Delta S_n\|^2  \right] \right)\\
&\qquad \stackrel{(ii)}{\leq} \cosh(\rho\|S_{n - 1}\|) \left(1 + \rho^2\beta^2\const_p(\B, \rho) \lambda_n^{p} (v + \|\mu - \zhat_n\|^p)\right) \\
&\qquad \stackrel{(iii)}{\leq} \cosh(\rho\|S_{n - 1}\|)\exp\left\{\rho^2\beta^2\const_p(\B, \rho) \lambda_n^{p} (v + \|\mu - \zhat_n\|^p)\right\},
\end{align*}
where (i) is obtained in view of $\int_0^1 (1-\theta) e^{a\theta} d\theta = \frac{e^a - a - 1}{a^2}$, (ii) is obtained from \eqref{eq:Sn2_bound} (and also using that $\beta=1$ in a Hilbert space), and (iii) follows from $1 + u \leq e^u$ for all $u \in \R$. Since $n \geq 1$ was arbitrary, rearranging yields that the process defined by $
\cosh(\rho\|S_n\|)\exp\left\{- \rho^2\beta^2\const_p(\B) G_n\right\}$
is a nonnegative supermartingale. Noting that $\frac{1}{2}\exp\left(\rho\|S_n\|\right) \leq \cosh(\rho\|S_n\|)$ yields the claimed result.
\end{proof}

\subsection{Step 3: Bounding $M_n(\lambda^{(n)}, \rho, \zhat^{(n)})$}
We now combine Lemma~\ref{lem:trunc_error} and Proposition~\ref{prop:banach} to write down an explicit form for the  supermartingale $(M_n)_{n \geq 0} \equiv (M_n(\lambda^{(n)}, \rho, \zhat^{(n)}))_{n \geq 0}$ in \eqref{eq:Mn}. 

\begin{lemma}
\label{lem:sub_p}
Let $(X_n)_{n \geq 1}$ and $(\B,\|\cdot\|)$ be as in Proposition~\ref{prop:banach}. Then, the process $(M_n)_{n \geq 0} \equiv (M_n(\lambda^{(n)}, \rho, \zhat^{(n)}))_{n \geq 0}$ defined by
\[
M_n := \frac12 \exp\left\{\rho\left\|\estimator_n - \mu \sum_{m \leq n }\lambda_m\right\| - (\rho^2\beta^2\const_p(\B) + \rho K_p2^{p-1})G_n \right\},
\]
is bounded above by a nonnegative supermartingale with initial value $1$.
\end{lemma}
\begin{proof}
Recall that $\wt{\mu}_n(\lambda) = \E_{n-1}[\Trunc(\lambda Y_n)Y_n] + \zhat_n$. 
Applying the triangle inequality twice and Lemma~\ref{lem:trunc_error} once, we obtain
\begin{align*}
&\left\|\estimator_n - \mu\sum_{m\leq n}\lambda_m\right\| 
\leq \left\|\estimator_n - \sum_{m \leq n} \lambda_m  \wt{\mu}_m(\lambda_m)\right\| + \sum_{m\leq n}\lambda_m \| \wt{\mu}(\lambda_m) - \mu\| \\ 
&\qquad \leq \left\|S_n\right\| + K_p2^{p-1}\sum_{m\leq n} \lambda_m^p(v + \|\mu-\zhat_m\|^p) = \left\|S_n\right\| + K_p2^{p-1}G_n.
\end{align*}
Therefore, 
\begin{align*}
    M_n &= \frac12 \exp\bigg\{ \rho\bigg\|\estimator_n - \mu \sum_{m\leq n} \lambda_m\bigg\| - (\rho^2\beta^2\const_p(\B) + \rho K_p2^{p-1})G_n\bigg\} \\
    &\leq \frac12  \exp\left\{\rho\|S_n\| + \rho K_p 2^{p-1} G_n  - (\rho^2 \beta^2\const_p(\B) + \rho K_p2^{p-1})G_n\right\}  \\
    &= \frac12 \exp\left\{ \rho \|S_n\| - \rho^2\beta^2\const_p(\B)G_n\right\},
\end{align*}
which is itself upper bounded by a nonnegative supermartingale with initial value 1 by Proposition~\ref{prop:banach}. 
\end{proof}

We are finally ready to prove Theorem~\ref{thm:template}, which follows as a consequence of the following result. 

\begin{prop}
\label{prop:template} 
Let $(\B,\|\cdot\|)$ satisfy Assumption~\ref{ass:smooth} and $(X_n')_{n\geq 1}$ satisfy Assumption~\ref{ass:mean} with respect to some filtration $(\calG_n)_{n\geq 0}$. Suppose 
$\zhat$ is $\calG_0$-measurable and there exists some function $r:(0,1)\to\Re_{\geq 0}$ such that, for any $\delta\in(0,1]$, 
\begin{equation}
    \P(\|\mu - \zhat\| \geq r(\delta))\leq \delta. 
\end{equation}
Fix any $\delta \in (0, 1]$. Let $\delta_1,\delta_2>0$. Then, for any $\rho, \lambda>0$, with probability $1-\delta$, simultaneously for all $n\geq 1$, we have: 
\begin{multline*}
\P\left(\left\| \est_n - \mu\right\| \leq \lambda^{p-1}(\rho\beta^2\const_p(\B) + K_p2^{p-1}) (v + r(\delta_2)^p) + \frac{\log(2/\delta_1)}{\lambda \rho n}\right) \\
\geq 1 - \delta_1 + \delta_2, 
\end{multline*}
where $\est_n = \frac{1}{n}\sum_{m\leq n}\{\Trunc(\lambda(X_m' - \zhat))(X_m' - \zhat) + \zhat\}$. 
\end{prop}
\begin{proof}
    Let $B_1 = \{\exists n:  M_n \geq 2/\delta_1\}$ where $(M_n)$ is as in Lemma~\ref{lem:sub_p}. By Ville's inequality (Section~\ref{sec:prelims}), $\P(B_1)\leq \delta_1$. Let $B_2 = \{\|\mu - \zhat\| \geq r(\delta_2)\}$. By assumption, $\P(B_2)\leq \delta_2$. Set $B = B_1\cup B_2$ so that $\P(B)\leq \delta$. We take the sequence of predictable values $(\lambda_n)$ in Lemma~\ref{lem:sub_p} to be constant and set $\lambda_n = \lambda>0$ for all $n$. 
    On the event $B^c$ we have $\log(M_n)\leq \log(2/\delta_1)$ for all $n\geq 1$. That is, with probability $1 - \delta$, 
    \begin{equation}
    \label{eq:proof-thm-est}
        \rho\| \estimator_n - \lambda n \mu\| \leq (\rho^2 \beta^2\const_p(\B) + \rho K_p2^{p-1})G_n + \log(2/\delta_1),
    \end{equation}
    and 
    \begin{equation}
    \label{eq:proof-thm-Gn}
        G_n = n\lambda^p(v + \|\mu - \zhat\|^p) \leq n\lambda^p  (v + r(\delta_2)^p).  
    \end{equation}
    Substituting \eqref{eq:proof-thm-Gn} into \eqref{eq:proof-thm-est} and dividing both sides by $n\lambda\rho $ gives the desired result. 
\end{proof}

\begin{proof}[Proof of Theorem~\ref{thm:template}]
Given $(X_n)_{n\geq 1}$ as in the statement of Theorem~\ref{thm:template} apply Proposition~\ref{prop:template} with
$X_n' = X_{n+k}$ and $\calG_n = \calF_{n+k}$ for all $n\geq 0$. Finally, take $\rho = \beta^{-1}$. 
\end{proof}

\section{Law of the Iterated Logarithm Rates}
\label{sec:lil}

In the previous section, we derived a time-uniform, \bd{} inequality that controlled (with high probability) the deviation between a truncated mean estimator and the unknown mean. This inequality was parameterized by a scalar/truncation level $\lambda$, which, when optimized appropriately, could guarantee a width of order $\beta v^{1/p}(\log(1/\delta)/n)^{(p - 1)/p}$ (up to known, lower order terms) with probability at least $1 - \delta$ for a preselected sample size $n$. However, in many settings, one may not know a target sample size in advance and may wish to observe the data sequentially and stop adaptively at a data-dependent stopping time. 

To generalize our bound to this setting, we use a technique known as \emph{stitching}~\citep{howard2021time}. Stitching involves breaking ``intrinsic time'' (governed by the number of observed samples) into geometrically spaced epochs. In each epoch, we appropriately deploy and optimize Theorem~\ref{thm:template}. By carefully applying a union bound over these epochs, we obtain time-uniform rates of convergence that match the tightness of Corollary~\ref{cor:opt-lambda} up to a doubly-logarithmic multiplicative factor in the sample size.

In more detail, we break the number of observed samples into geometric buckets of the form $[2^j, 2^{j + 1})$ for $j \geq 1$. Given a failure probability $\delta \in (0, 1)$ and a ``stitching function'' $h : \R_{\geq 0} \rightarrow \R_{\geq 0}$ which satisfies $\sum_{j\geq 1}1/h(j) \leq 1$, we allocate a total failure probability of $\delta/h(j)$ to the bound constructed in bucket $j$. Given this failure probability, in each bucket (i.e.\ for $n \in [2^j, 2^{j + 1})$) we bound the performance of the estimator
\[
\wh{\mu}_n(j) := \frac{1}{n - k(j)}\sum_{m = k(j)+1}^n \left\{\Trunc(\lambda_j (X_m - \zhat_j))(X_m - \zhat_j)+ \zhat_j\right\}.
\]
Here, $k(j)$ denotes the number of samples to use for the naive mean estimate in the $j$th epoch, $\zhat_j \equiv \zhat_j(X_1, \dots, X_{k(j)})$ denotes the naive mean estimate using the first $k(j)$ samples, and $\lambda_j$ is some to-be-determined truncation level that will allow for rapid convergence. By separately bounding the deviations of the above estimators over each epoch, we can thus gauge the performance of the overall estimator
\[
\wh{\mu}_n := \sum_{j = 1}^\infty \wh{\mu}_n(j)\mathbbm{1}\{n \in [2^j, 2^{j + 1})\}.
\]
We formalize this in the following theorem.

\begin{theorem}
\label{thm:stitching}

Let $(\B, \|\cdot\|)$ satisfy Assumption~\ref{ass:smooth} and let $(X_n)_{n \geq 1}$ satisfy Assumption~\ref{ass:mean}. Let $h : \R_{\geq 0} \rightarrow \R_{\geq 0}$ be a stitching function satisfying $\sum_{n = 1}^\infty h(n)^{-1} \leq 1$ and let $k : \N \rightarrow \N$ be an arbitrary increasing function satisfying $k(\ell)\leq 2^{\ell-1}$ 
for all $\ell \geq n_0$, for some $n_0 \in \N$. Let $(\zhat_j)_{j \geq 1}$ 
satisfy Assumption~\ref{ass:pilot}. 
Then, for any confidence parameter $\delta \in (0, 1)$, we have with probability at least $1 - \delta$,
\[
\|\wh{\mu}_n - \mu\| \leq \beta \mathfrak{B}_p(\B)(v + r_n^p)^{1/p}\left(\frac{\log(4h(\lfloor\log_2(n)\rfloor))}{n - k(\lfloor \log_2(n)\rfloor)}\right)^{(p - 1)/p}.
\]
simultaneously for all $n \geq n_0$, where $\wh{\mu}_n$ is as defined above, 
\begin{align*}
r_n &:= r\left(h(\lfloor\log_2(n)\rfloor)^{-1}\delta/2, k(\lfloor \log_2(n)\rfloor)\right) \text{~ and} \\ 
\mathfrak{B}_p(\B) &:= [2^{(p - 1)/p} + 2^{1/p}](\const_p(\B) + \beta^{-1}K_p2^{p - 1})^{1/p}.
\end{align*}

\end{theorem}

For instance, if we take $h(j) = \frac{\pi^2}{6}(j+1)^2$,  $k(j) = 2^{j-1}$, and obtain $r(\delta, k)$ using Lemma~\ref{lem:mean_concentration}, then we have 
\[
r_n =  2\beta \left(\frac{ 2v\cdot h(\lfloor \log_2(n)\rfloor)}{\delta \cdot 2^{(\lfloor \log(n)\rfloor-1)(p-1)}}  \right)^{1/p} 
< 2\beta \left(\frac{ 4v\cdot 4^{p-1} (\log_2(n)+1)^2}{\delta \cdot n^{p-1}}  \right)^{1/p} \xrightarrow{n\to\infty}0. 
\]
Likewise, we can give similar convergence guarantees if $\zhat_j$ is taken to be, say, a geometric median-of-means estimator. We now prove Theorem~\ref{thm:stitching}.

\begin{proof}[Proof of Theorem~\ref{thm:stitching}]
Set $\delta_j := \delta/h(j)$ and set 
\[\const_p(j) := \left(\beta\const_p(\B) + 2^{p - 1}K_p\right)\left(v + r(\delta_j/2, k(j)\right),\] for convenience. For any $j \geq 1$, define the tuning parameter $\lambda_j$ as
\[
\lambda_j := \left(\frac{\beta\log(4/\delta_j)}{2^j\const_p(j)}\right)^{1/p}.
\]

For $n \geq n_0$, $j \geq 1$, define the confidence bounds 
\[
W(n, j) := \lambda_j^{p - 1}\const_p(j) + \frac{\beta\log(4/\delta_j)}{\lambda_j (n - k(j))}
\]

If we define the ``bad'' event 
\[B := \bigcup_{j\geq 1}B_j, \quad B_j := \left\{\exists n \geq n_0, n\in [2^j,2^{j+1})] : \|\wh{\mu}_n - \mu\| \geq W(n, j)\right\},\] 
and apply Theorem~\ref{thm:template} with $k = k(j)$, $\zhat_k = \zhat_j$, $\lambda = \lambda_j$, and $\delta_1 = \delta_2 = \delta_j/2$, we have $\P(B_j)\leq \delta_j$. Hence, 
\begin{align*}
\P(B) & \leq \sum_{j\geq 1} \P(B_j)
\leq \sum_{j\geq 1}\delta_j  = \sum_{j \geq 1}h(j)^{-1}\delta \leq \delta.
\end{align*}

We now want to show the target time-uniform bound on the sequence of estimators $(\wh{\mu}_n)_{n \geq n_0}$. Operating on the ``good'' event $B^c$, letting $n \geq n_0$ be arbitrary, and letting $j$ denote the index satisfying $n \in [2^j, 2^{j + 1})$, we have
\begin{align*}
&\|\wh{\mu}_n - \mu\|  = \|\wh{\mu}_n(j) - \mu\| \\
&\qquad \leq W(n, j) = \lambda_j^{p - 1}\const_p(j) + \frac{\log(4/\delta_j)}{\lambda_j (n - k(j))} \\
&\qquad = \const_p(j)^{1/p}\left(\beta\log(4/\delta_j)\right)^{(p - 1)/p}\left[\left(\frac{1}{2^{j}}\right)^{(p - 1)/p} + \frac{2^{j/p}}{n - k(j)}\right] \\
&\qquad = \const_p(j)^{1/p}\left(\frac{\beta\log(4/\delta_j)}{n - k(j)}\right)^{(p - 1)/p}\left[\left(\frac{n - k(j)}{2^{j}}\right)^{(p - 1)/p} + \left(\frac{2^j}{n - k(j)}\right)^{1/p}\right] \\
&\qquad \leq \const_p(j)^{1/p}\left(\frac{\beta\log(4/\delta_j)}{n - k(j)}\right)^{(p - 1)/p}\left[2^{(p - 1)/p} + 2^{1/p}\right],
\end{align*}
where the last inequality follows from the fact that $2^{j - 1} \leq n - k(j) \leq 2^{j + 1}$, since $n\geq n_0$ implies $k(j) \leq 2^j/2 \leq n/2$.

Lastly, recalling that $\delta_j = h(j)^{-1}\delta$ and that, by definition, $j = \lfloor \log_2(n)\rfloor$, we have
\[
\log(4/\delta_j) \leq \log(4h(\lfloor \log_2(n)\rfloor)/\delta).
\]
Putting everything together and expanding the definition of $\const_p(j)$ yields that, with probability at least $1 - \delta$, simultaneously for all $n \geq n_0$
\begin{align*}
&\|\wh{\mu}_n - \mu\| \\
&\quad \leq [2^{\frac{p - 1}{p}} + 2^{\frac{1}{p}}]\left((\beta \const_p(\B) + K_p2^{p - 1})(v + r_n^p)\right)^{\frac{1}{p}}\left(\frac{\beta\log(4h(\lfloor\log_2(n)\rfloor))}{n - k(j)}\right)^{\frac{p - 1}{p}} \\
&\quad = \beta\mathfrak{B}_p(\B)( v+ r_n^p)^{\frac{1}{p}} \left(\frac{\log(4h(\lfloor\log_2(n)\rfloor))}{n - k(\lfloor \log_2(n)\rfloor)}\right)^{\frac{p - 1}{p}},
\end{align*}
thus proving the claimed result.
\end{proof}

\section{Bound Comparison and Simulations}
\label{sec:experiments}
In the above sections, we argued that the truncated mean estimator, when appropriately optimized, could obtain a distance from the true mean of order $\beta v^{1/p}\left(\log(1/\delta)/n\right)^{(p - 1)/p}$ with high probability. In particular, this rate matched that of the geometric median-of-means estimator due to \citet{minsker2015geometric}, per our discussion in Appendix~\ref{sec:gmom-banach}. In this section, we study the empirical instead of theoretical performance of our bounds and estimator.

\paragraph{Comparing Tightness of Bounds}

In Figure~\ref{fig:bound_comp}, we compare the confidence bounds obtained for our truncation-based estimators optimized for a fixed sample size (Theorem~\ref{thm:template}) against other bounds in the literature. Namely, we compare against geometric median-of-means~\citep{minsker2015geometric}, the sample mean, and (in the case a shared covariance matrix exists for observations) the tournament median-of-means estimator~\citep{lugosi2019sub}. We plot the natural logarithm of the bounds against the logarithm base ten of the sample sizes $n$ for $n \in [10^2, 10^{10}]$ and for $p \in \{1.25, 1.5, 1.75, 2.0\}$. We assume $\delta = 10^{-4}$ and $v = 1$. For our truncation-based estimates, we assume $k = \lfloor n/10\rfloor$ samples are used to produce the initial mean estimate and the remaining $n - k$ are used for the final mean estimate. We plot the resulting bounds for when the initial mean estimate is either computed using the sample mean or geometric median-of-means.
For the tournament median-of-means estimate, we assume observations take their values in $\R^d$ for $d = 100$, and that the corresponding covariance matrix is the identity $\Sigma = I_d/d$.

As expected, all bounds have a slope of $-(p - 1)/p$ when $n$ is large, indicating equivalent dependence on the sample size. For all values of $p$, the truncation-based estimator using geometric median-of-means as an initial estimate obtains the tightest rate once moderate sample sizes are reached ($n = 10^4$ or $n = 10^5$). When $p \in \{1.25, 1.5\}$, much larger sample sizes are needed for truncation-based estimates with a sample mean initial estimate to outperform geometric median means (needing $\geq 10^{10}$ samples for $p = 1.25$).
For $p=2.0$ (i.e., finite variance) the tournament median-of-means estimate, despite achieving optimal sub-Gaussian dependence on $\lambda_{\max}(\Sigma)$ and $\Tr(\Sigma) = v$, performs worse than even the naive mean estimate. This is due to prohibitively large constants. 
These plots suggest that the truncation-based estimate is a practical and computationally efficient alternative to approaches based on median-of-means.

\begin{figure}
    \centering
    \begin{subfigure}{0.4\textwidth}
    \centering
    \includegraphics[width=\linewidth]{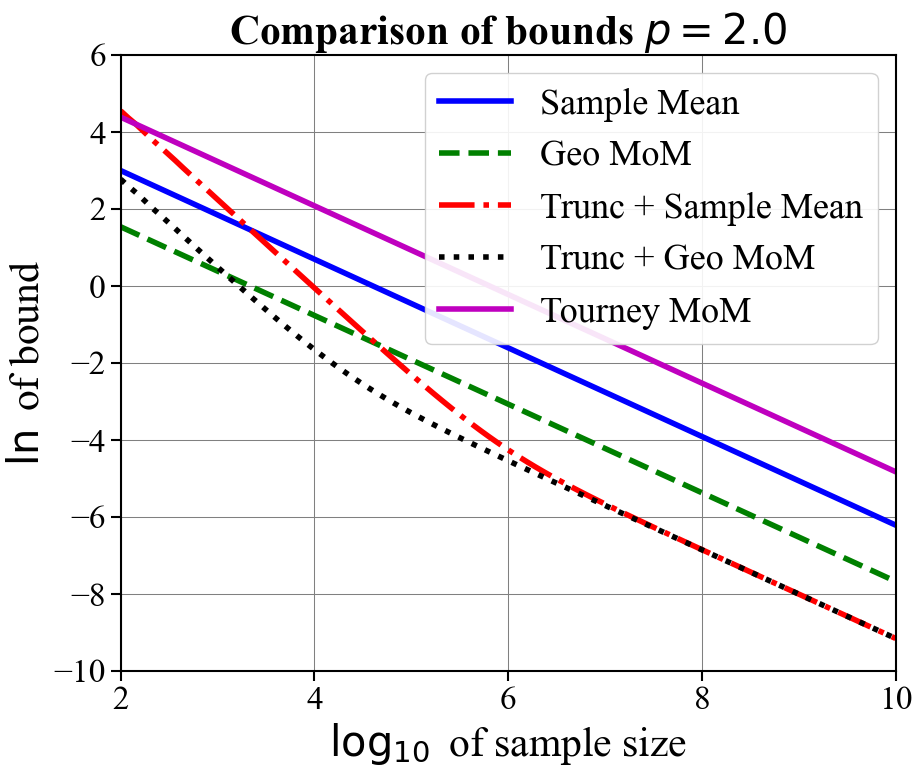}
    \caption{$p = 2.0$}
    \label{fig:bound_comp:p=2.0}
    \end{subfigure}
    \begin{subfigure}{0.4\textwidth}
    \centering
    \includegraphics[width=\linewidth]{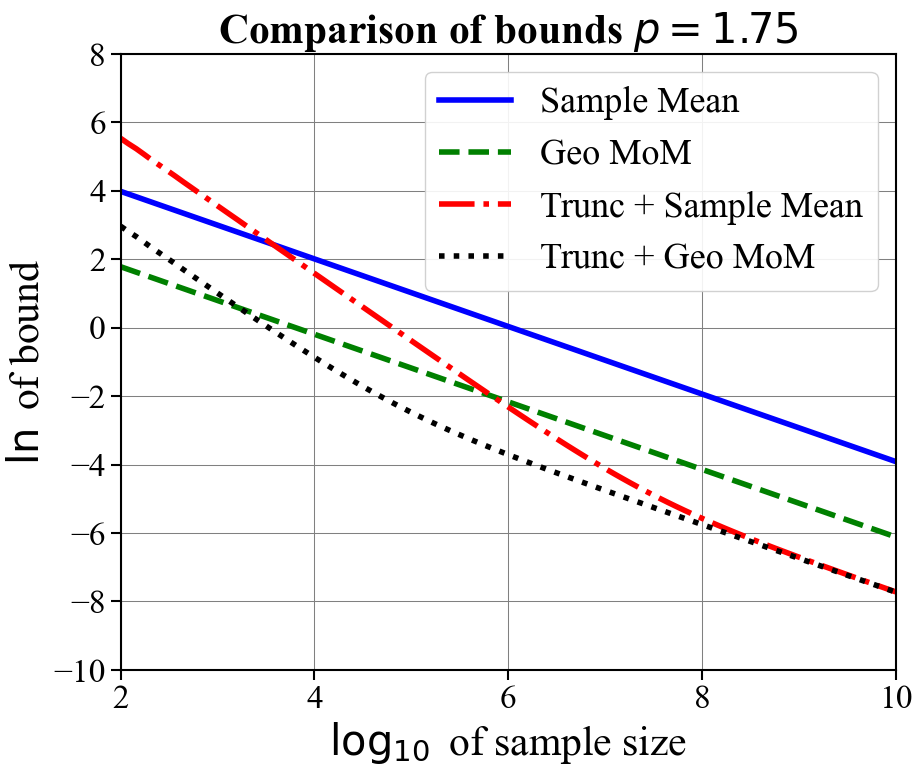}
    \caption{$p=1.75$}
    \label{fig:bound_comp:p=1.75}
    \end{subfigure}
    \begin{subfigure}{0.4\textwidth}
        \centering
    \includegraphics[width=\linewidth]{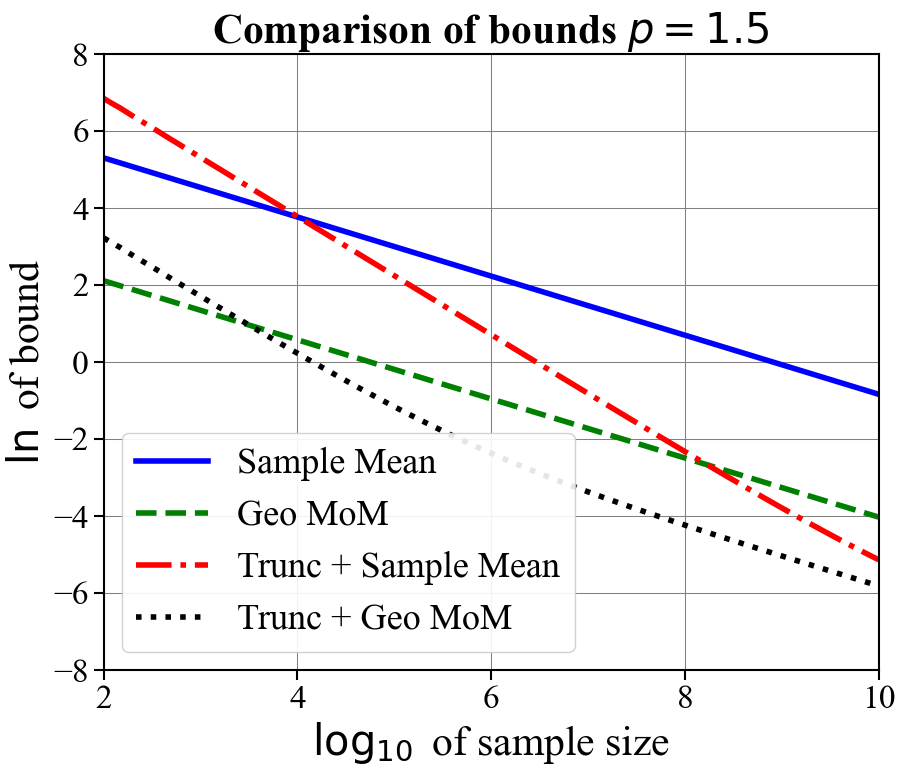}
    \caption{$p=1.5$}
    \label{fig:bound_comp:p=1.5}
    \end{subfigure}
        \begin{subfigure}{0.4\textwidth}
        \centering
    \includegraphics[width=\linewidth]{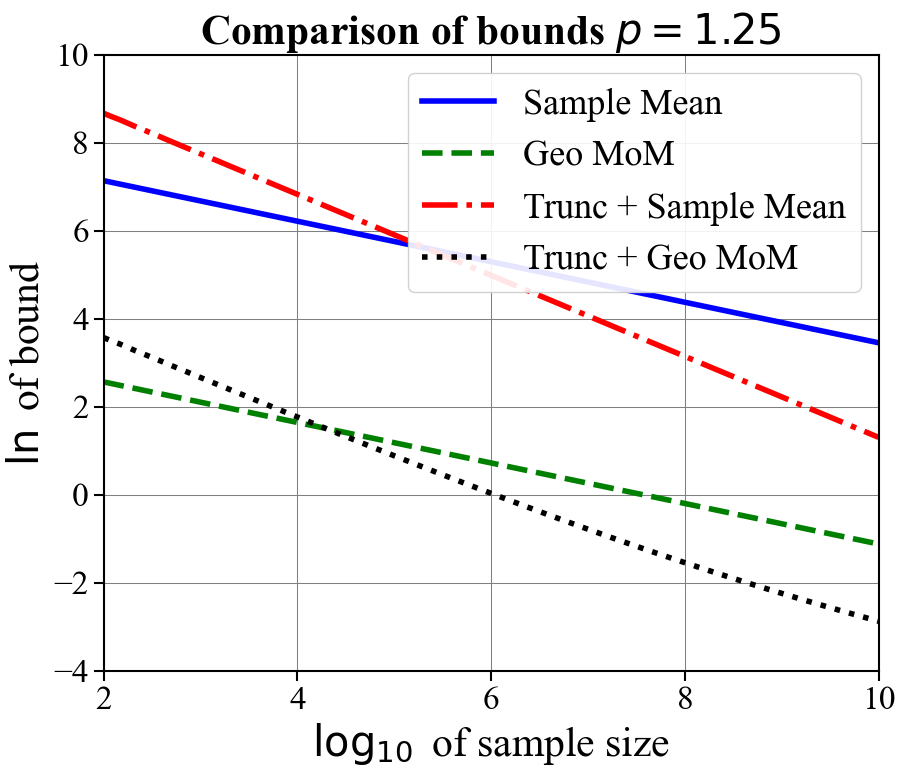}
    \caption{$p=1.25$}
    \label{fig:bound_comp:p=1.25}
    \end{subfigure}
    \caption{For $p \in \{1.25, 1.5, 1.75\}$, we plot the tightness of optimized bounds associated with the sample mean, geometric median-of-means (Geo-MoM), truncation with initial sample mean estimate, and truncation with initial Geo-MoM estimate. We assume $n \in [10^2, 10^{10}]$, $v = 1.0$, $\delta=10^{-4}$, and $k = n/10$. In the case $p=2.0$, we assume a shared covariance matrix $\Sigma$ exists so we can plot the tournament median-of-means bounds assuming $\lambda_{\max}(\Sigma) = v/d$ and $d = 100$.  }
    \label{fig:bound_comp}
\end{figure}

\paragraph{Performance of Estimators on Simulated Data}

In Figure~\ref{fig:both}, we examine the performance of the various mean estimators by plotting the distance between the estimates and the true mean. We explore both finite and infinite dimensional spaces. To do this, we sample $n = 10^5$ i.i.d.\ samples $X_1, \dots, X_n$ in two different ways. First, for our finite-dimensional experiments, we sample i.i.d.\ directions $U_1, \dots, U_n \sim \mathrm{Unif}(\S^{d - 1})$ from the $d = 10$ dimensional unit sphere, as well as i.i.d.\ magnitudes $Y_1, \dots, Y_n \sim \mathrm{Pareto}(a)$ from the Pareto II (or Lomax) distribution with $a= 1.75$.\footnote{If $Y \sim \mathrm{Pareto}(a)$, then $Y$ has inverse polynomial density $f_Y(x) = a(1 + x)^{-(a + 1)}$.} The learner observes $X_1 = Y_1 \cdot U_1, \dots, X_n = Y_n \cdot U_n$, and constructs either a geometric median estimate, a sample mean estimate, or a truncated mean estimate.
Second, we consider the infinite-dimensional separable Hilbert space $\ell^2(\mathbb{N})$ with an orthonormal basis $\{ e_1, e_2, \ldots \}$. We also sample i.i.d.\ magnitudes $Y_1, \dots, Y_n \sim \mathrm{Pareto}(a)$ from the Pareto II distribution with $a = 1.75$, but we now sample i.i.d.\ directions $\xi_1e_{I_1}, \ldots, \xi_ne_{I_n}$, where the $I_i$ follow a Poisson distribution with rate parameter $\zeta = 5$ and the $\xi_i$ follow a Rademacher distribution, i.e.\ they satisfy $\P(\xi = 1) = \P(\xi = -1) = 1/2$, which is to enforce zero mean random vectors. The learner observes $X_1 = Y_1 \xi_1 e_{I_1}, \ldots, X_n = Y_n \xi_n e_{I_n}$.

To compute the number of folds for geometric median-of-means, we follow the parameter settings outlined in \citet{minsker2015geometric} and assume a failure probability of $\delta = 10^{-4}$ (although we are not constructing confidence intervals, the failure probability guides how to optimize the estimator). See Appendix~\ref{sec:gmom-banach} for further discussion on this estimator. Once again, we consider the truncated mean estimator centered at both the sample mean and a geometric median-of-means estimate. We always use $k = \lfloor \sqrt{n}\rfloor$ samples to construct the initial estimate, and $\lambda = 0.5$. We defer further experiments with other hyperparameter choices to Appendix~\ref{sec:further_experiments}.

\begin{figure}
    \centering
    \begin{subfigure}{0.45\textwidth}
    \centering
    \includegraphics[width=\linewidth]{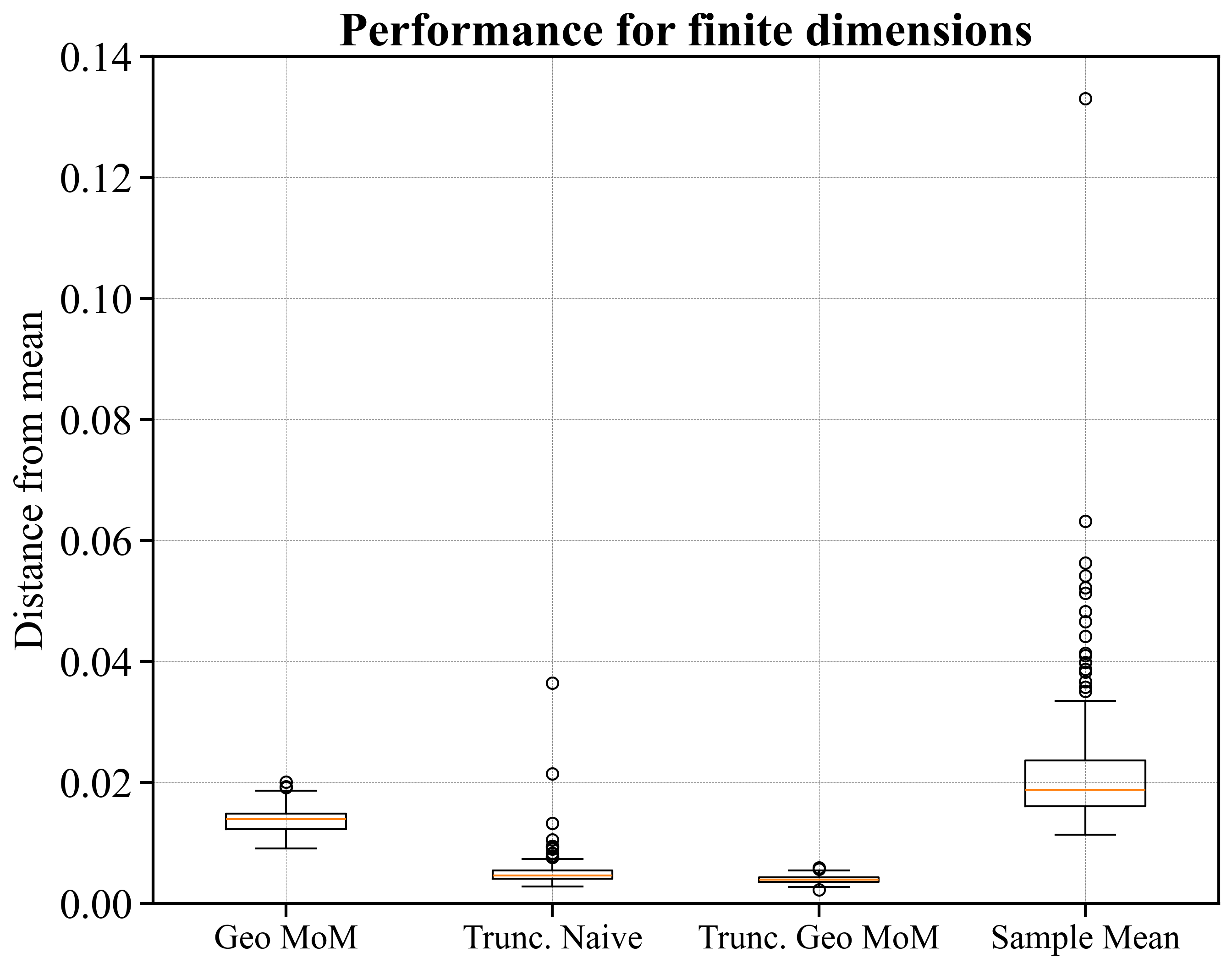}
    \caption{Dimension of the space: $10$.}
    \label{fig:both:rd}
    \end{subfigure}
    \begin{subfigure}{0.45\textwidth}
    \centering
    \includegraphics[width=\linewidth]{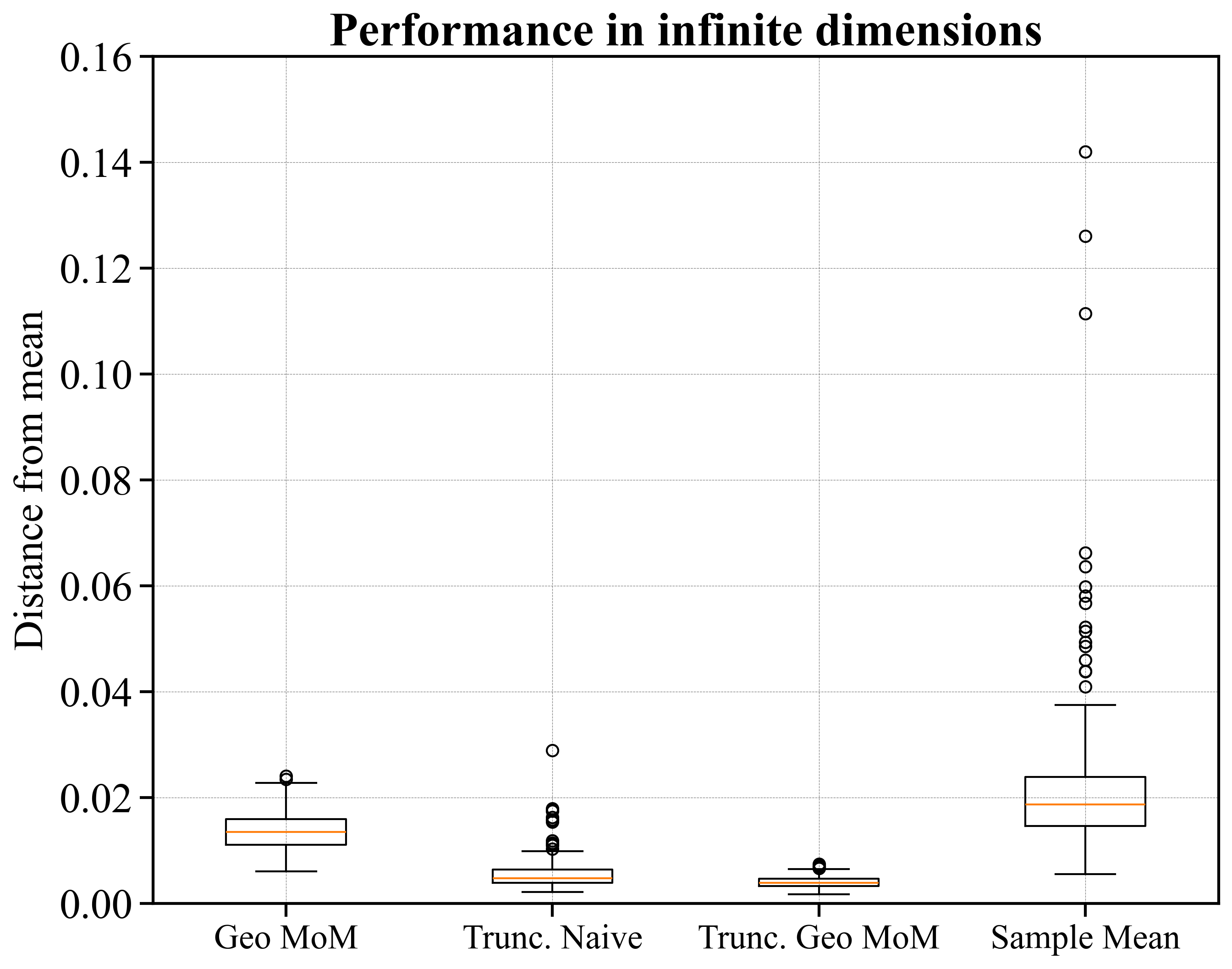}
    \caption{Dimension of the space: $\infty$.}
    \label{fig:est_dist:lam=5e-3}
    \end{subfigure}
    \caption{We compare the empirical distributions of distance between the mean estimate and the true mean for a variety of estimators: the sample mean, the geometric median-of-means~\citep{minsker2015geometric}, and our truncation-based estimators from Theorem~\ref{thm:template} (with either the sample mean or geometric median-of-means for the initial mean estimates). We generate $n =10^6$ i.i.d.\ samples in $\R^{10}$ or $\ell^2(\N)$ as outlined above, and use $k = \lfloor\sqrt{10^6}\rfloor$ samples to construct initial mean estimates. We compute these estimates on 250 runs. We highlight the comparable performance of our truncation-based methods and the geometric median-of-means, while the sample mean suffers substantially from outliers.}
    \label{fig:both}
\end{figure}

We construct these estimators over 250 independent runs and then construct box and whisker plots summarizing the empirical distance between the estimators and the true mean. The boxes have as a lower bound the first quartile $Q1$, in the middle the sample median $M$, and at the top the third quartile $Q3$. The whiskers of the plot are given by the largest and smallest point falling within $M \pm 1.5\times(Q3 - Q2)$, respectively. All other points are displayed as outliers. 

As expected, the sample mean suffers heavily from outliers. In contrast, our truncation-based estimators show comparable power to the geometric median-of-means in both the finite and infinite dimensional examples considered. As before, we view our estimator as a viable alternative to geometric median-of-means, namely due to the fact that it is computationally lightweight, operates under martingale dependence, and provides time-uniform guarantees.

\section{Summary and Conclusion}
In this work, we presented a novel analysis of a simple truncation/threshold-based estimator of a heavy-tailed mean in smooth Banach spaces, strengthening the guarantees on such estimators that currently exist in the literature. In particular, we allow for martingale dependence between observations, move beyond finite-dimensional settings, and replace the assumption of finite non-central second moment with a finite central $p$-th moment for $1< p\leq 2$ (thus making the estimator translation invariant). Our bounds are also time-uniform, meaning they hold simultaneously for all sample sizes. We provide both a \bd{} inequality that can be optimized for a particular sample size (but remains valid at all times), and a bound whose width shrinks to zero at an iterated logarithm rate. Experimentally, our estimator performs quite well compared to more computationally intensive methods such as geometric median-of-means, making it an appealing choice for practical problems. 

There are several important open directions related to the work in this paper. We note that in order to use the confidence interval provided by Corollary~\ref{cor:opt-lambda}, we needed to select the truncation level $\lambda$ to explicitly depend on the moment bound $v$. This is in contrast to point estimators such as geometric median-of-means, whose parameter need not make explicit use of moment bounds to obtain the same rate. Thus, it would be interesting to see if one could develop a point-estimator analogue of the estimator considered in this paper, or perhaps an estimator that could automatically adapt to an unknown $p$th moment bound $v$.
Further, in the setting of finite-dimensional Euclidean space where $p = 2$ and the data are i.i.d.\ with shared covariance matrix $\Sigma$, estimators have been constructed (see \citet{lugosi2019sub, lugosi2019mean}) obtaining a high-probability rate of 
\[
\|\wh{\mu}_n - \mu\| \leq C\left(\sqrt{\frac{\Tr(\Sigma)}{n}} + \sqrt{\frac{\lambda_{\max}(\Sigma)\log(1/\delta)}{n}}\right),
\]
where $C > 0$ denotes some absolute constant. Noting that $\Tr(\Sigma) = \E\|X_n\|^2$ and $\lambda_{\max}(\Sigma) = \sup_{\nu \in \S^{d - 1}} |\langle \nu, X_n\rangle|^2$, we believe that the natural equivalent of the trace and maximal eigenvalue of the covariance matrix in settings where $p < 2$ and $(\B, \|\cdot\|)$ is a Banach space would be $v := \E_{n - 1}\|X_n\|^p$ and $w := \sup_{\nu \in \S^\ast}\E_{n - 1}|\langle \nu, X_n\rangle|^p$, where $\S^\ast$ denotes the unit ball in the dual space $(\B^\ast, \|\cdot\|_\ast)$. Surprisingly, it seems as if no existing works (including the present one and \citet{minsker2015geometric}) have obtained a ``separation of rates'' as has been accomplished in the setting $p = 2$ in Euclidean space. Perhaps it is possible to obtain a high-probability guarantee of the form
\[
\|\wh{\mu}_n - \mu\| \leq C_p\left(\beta v^{1/p}\left(\frac{1}{n}\right)^{(p - 1)/p} + \beta w^{1/p}\left(\frac{\log(1/\delta)}{n}\right)^{(p - 1)/p}\right)
\]
for some appropriately defined estimator $\wh{\mu}_n$ of $\mu$ and constant $C_p > 0$. We leave obtaining such a bound as interesting future work.

Lastly, our results only hold for $2$-smooth Banach spaces, and in particular do not hold for $\ell^\alpha$ sequence spaces and $L^\alpha$ function spaces for $\alpha \in (1, 2)$. To the best of our knowledge, there do not seem to be concentration of measure results for bounded observations analogous to \citet{pinelis1992approach, pinelis1994optimum} in the setting of $(q, \beta)$-smooth spaces for $q < 2$. Such results would naturally be a requirement to extend our results to these spaces. In these settings, $\sqrt{n}$-rates of concentration are no longer possible in general. For instance, if $(\epsilon_n)$ represents a sequence of i.i.d.\ Rademacher random variables and $e_i$ represents the sequence in $\ell^\alpha$ whose $i$th element is one and all other elements are zero, the random variables $X_n := \epsilon_n e_n$ form a martingale difference sequence with bounded increments, and one has  $\|\wh{\mu}_n\| = \left(\frac{1}{n}\right)^{(q - 1)/q}$ almost surely, where here $\wh{\mu}_n$ denotes the sample mean. It is thus difficult to speculate what appropriate widths for confidence intervals would look like in this setting assuming the existence of a finite, bounded $p$th moment.

\subsection*{Acknowledgments}
DMT acknowledges the support of a fellowship from `la Caixa' Foundation (ID 100010434), code number LCF/BQ/EU22/11930075.
BC and AR acknowledge support from NSF grant IIS-2229881.
BC was supported in part by the NSERC PGS D program, grant no. 567944.

\bibliographystyle{plainnat}
\bibliography{bib}

\appendix
\numberwithin{equation}{section}
\renewcommand{\thefigure}{\thesection.\arabic{figure}}
\section{Extended Related Work and Discussing Laws of Large Numbers}
\label{sec:ext-related-work}

{We now discuss additional related work that, while less directly related to the topic of the paper, is nonetheless important to mention to contextualize our results.

\paragraph{Sequential concentration under martingale dependence} 
Time-uniform concentration bounds, or concentration inequalities that are valid at data-dependent stopping times, have been the focus of significant recent attention~\citep{howard2020time,howard2021time, whitehouse2023time}. Such results are often obtained by identifying an underlying nonnegative supermartingale and then applying Ville's inequality~\citep{ville1939etude}, a strategy that allows for martingale dependence quite naturally. This approach is also used here. 
\citet{wang2023catoni} extend Catoni's M-estimator to handle both infinite variance and martingale dependence in $\Re$, while  \citet{chugg2023time} give a sequential version of the truncation estimator in $\Re^d$, though they require a central moment assumption and finite variance.  The analyses of both \citet{catoni2018dimension} and \citet{chugg2023time} rely on so-called ``PAC-Bayes'' arguments~\citep{catoni2007pac,chugg2023unified}.  Intriguingly, while we analyze a similar estimator, our analysis avoids such techniques and is much closer in spirit to Pinelis-style arguments~\citep{pinelis1992approach,pinelis1994optimum}.  

\citet{howard2020time, howard2021time} provide a general collection of results on time-uniform concentration for scalar processes, which in particular imply time-uniform concentration results for some heavy-tailed settings (e.g.\ symmetric observations). Likewise, \citet{whitehouse2023time} provide a similar set of results in $\R^d$. While interesting, we note that these results differ from our own in that they are \textit{self-normalized}, or control the growth of a process appropriately normalized by some variance proxy (there a mixture of adapted and predictable covariance). The results also don't apply when only a bound on the $p$th moment is known, and the latter set of results have explicit dependence on the ambient dimension $d$.

\paragraph{Concentration in Hilbert and Banach Spaces}
There are several results related to concentration in infinite-dimensional spaces. A series of works has developed self-normalized, sub-Gaussian concentration bounds in Hilbert spaces ~\citep{whitehouse2023sublinear, abbasi2013online, chowdhury2017kernelized} based on the famed method of mixtures~\citep{de2004self, de2007pseudo}. These results have not been extended to more general tail conditions.
Significant progress has been made on the concentration of bounded random variables in smooth and separable Banach spaces. 
\citet{pinelis1992approach, pinelis1994optimum} presented a martingale construction for bounded observations, thus enabling dimension-free Hoeffding and Bernstein inequalities. Dimension-dependence is replaced by the smoothness parameter of the Banach space, which for most practical applications (in Hilbert spaces, say) equals one. 
These results were strengthened slightly by \citet{howard2020time}. 
Recently, \citet{martinez2024empirical} gave an \emph{empirical}-Bernstein inequality in Banach spaces, also using a Pinelis-like construction.  Our work adds to this line by extending Pinelis' tools to the heavy-tailed setting.

\paragraph{Laws of Large Numbers in Banach Spaces}
There is a rich literature dating back to the 1970's focused on developing laws of large numbers (LLNs) and laws of the iterated logarithm (LILs) for random variables taking values in Banach spaces. \citet{mourier1953elements} proves early strong LLNs for independent random variables taking values in Banach spaces. \citet{hoffmann1976law} investigate connections between the type/smoothness of a Banach spaces, showing strong LLNs for independent random variables in Banach spaces of Rademacher types in $[1,2)$, and central limit theorems for spaces of Rademacher type 2. \citet{de1981inequalities} extend these results, proving Marcinkiewicz-Zygmund inequalities and Marcinkiewicz strong LLNs for independent observations, depending on the Rademacher type of the Banach space. \citet{assouad1974espaces} show Burkholder inequalities for martingales in smooth Banach spaces, and \citet{woyczynski1975lln, woyczynski2006asymptotic, dedecker2007convergence} extend these results to Marcinkiewicz strong LLNs.
Finally, \citet{cuny2015compact, cuny2017invariance} develop LILs in Banach spaces for martingales and stationary sequences satisfying Maxwell-Woodroofe conditions.

We view these results as complementary to our work and others on heavy-tailed mean estimation. The former results typically characterize the \textit{asymptotic convergence} for a \text{fixed joint distribution} over observations satisfying certain moment conditions.
The latter hold in \textit{finite-sample settings} in a way that is \textit{distribution-uniform} --- our stated high probability bounds hold for any sample size $n$, uniformly over all joint distributions that satisfy Assumption~\ref{ass:mean} below. 
One set of results does not imply the other and vice versa. 
In fact, while the asymptotic results focus mainly on the sample mean, \emph{it is now well known that even in scalar settings, the sample mean is suboptimal and it is necessary to significantly alter the estimator in order to achieve tight nonasymptotic concentration results}~\citep{catoni2012challenging,devroye2016sub}.
We provide further discussion and more illustrative examples at the end of Section~\ref{sec:results}. 

\subsection{A Comparison with Marcinkiewicz-Zygmund Laws of Large Numbers}
\label{sec:comparison}
We now briefly compare  Corollary~\ref{cor:opt-lambda} with the Marcinkiewicz-Zygmund Laws of Large Numbers discussed above. In what follows, let $(X_n)_{n \geq 1}$ denote a particular collection of random variables satisfying Assumption~\ref{ass:mean}.

The Marcinkiewicz-Zygmund law of large numbers present in \citet{woyczynski2006asymptotic} and \citet{dedecker2007convergence} yields that $n^{(p - 1)/p} \|\hat{\mu}_n - \mu\| \xrightarrow[n \rightarrow \infty]{\P} 0$, where here $\hat{\mu}_n := n^{-1}\sum_{k = 1}^n X_k$ is the sample mean. In particular, this implies that, for any $\delta > 0$, there is some sample size $N(\delta)$ such that 
\begin{equation}
\label{eq:ky-fan}
\P\left(\|\wh{\mu}_n - \mu\| \geq \frac{\delta}{n^{(p - 1)/p}}\right) \leq \delta, \qquad \forall n \geq N(\delta).
\end{equation}
While this appears to give better dependence on the failure probability $\delta$ (appearing directly as $\delta$ as opposed to $\log(1/\delta)$), this is in fact misleading. This is because the Marcinkiewicz-Zygmund LLN applies for a \textit{fixed joint distribution} over the sequence of observations $(X_n)_{n \geq 1}$ --- the sample size $N(\delta)$ in~\eqref{eq:ky-fan} is unknown and implicitly depends on the said joint distribution. 

Compare the above situation to another well known fact.
\citet{catoni2012challenging} showed that for any $n \geq 1$, failure probability $\delta < \frac{1}{2e}$, and variance $v >0$, there exists a sequence of i.i.d.\ mean zero random variables $X_1, \dots, X_n$ with $\Var[X_i]= v$ satisfying
\[
\P\left(|\wh{\mu}_n| \geq \sqrt{\frac{v}{\delta}}\left(1 - \frac{e\delta}{n}\right)^{n - 1}\right) \geq \delta,
\]
where again $\wh{\mu}_n$ denotes the sample mean. In particular, this means that one cannot non-asymptotically obtain the dependence on $\delta$ illustrated in~\eqref{eq:ky-fan}, uniformly over a class of distributions with variance bounded by $v$ (a situation we generalize to Banach spaces). We generalize this fact to arbitrary values of $p \in (1, 2]$ in Proposition~\ref{prop:lower_bd_sample_mean}, showing there exist i.i.d.\ mean zero random variables with $\E|X_i|^p = v$  satisfying the lower bound 
\[
\P\left(|\wh{\mu}_n| \geq \frac{e^{-1/p}v^{1/p}}{\delta^{1/p}n^{(p - 1)/p}}\right).
\]
This further emphasizes the subtleties involving non-asymptotic mean estimation/concentration, and why different tools are often needed, compared to limit theorems like laws of large numbers.

A remaining question is thus as follows: are our results optimal or unimprovable in some worst-case sense? Our bounds match known lower bounds provided by \citet{devroye2016subgaussian, lugosi2019mean} in the scalar setting   --- for any given mean estimator $\wh{\mu}_n$ (not necessarily the sample mean or the truncation-based estimator considered above), any failure probability $\delta \in (0, 1)$, any moment $1 \leq p \leq 2$, and any known bound $v$ on the $p$th central moment, there exists a collection of (scalar) i.i.d.\ random variables $X_1, \dots, X_n$ with some mean $\mu$ satisfying
\[
\P\left(\left|\hat{\mu}_n - \mu\right| > v^{1/p}\left(\frac{\log(2/\delta)}{n}\right)^{(p - 1)/p}\right) \geq \delta,
\]
which is precisely the rate (up to lower order terms and constants) provided by Corollary~\ref{cor:opt-lambda}. We formally state this lower bound in Proposition~\ref{prop:lower_bd_mean}. In the general setting of 2-smooth Banach spaces, our bounds obtain the same rate (i.e.\ dependence on $n$, $v$, $p$, and $\delta$) as the best known bounds in the literature (geometric median-of-means due to \citet{minsker2015geometric}). Given that any lower bound from the scalar setting also serves as a lower bound in Banach spaces, the dependence of our bounds on $n$ and $\delta$ cannot generally be improved. We hypothesize that the dependence on the smoothness parameter $\beta > 0$ can also not be improved, but leave constructing lower bounds for future work.

\edit{\section{Failure of the Sample Mean}
\label{sec:mean_failure}
}

\edit{
We briefly discuss the convergence of the sample mean under heavy-tailed observations in smooth Banach spaces. The following result assures us that the sample mean does in fact concentrate in such spaces, thus enabling us to use it as a pilot mean estimate in Theorem~\ref{thm:template}. The proof in Appendix~\ref{sec:proofs} leverages an argument based on decoupled tangent sequences, which may be of independent interest. The result is similar to those discussed in \citet{assouad1974espaces}, albeit with known, small constants.

\begin{lemma}[Sample Mean Concentration]
\label{lem:mean_concentration}
Let $(X_n)_{n \geq 1}$ be a sequence random variables satisfying Assumption~\ref{ass:mean} taking values in a Banach space satisfying Assumption~\ref{ass:smooth}. Then, for any $\delta \in (0, 1)$ and $n \geq 1$, we have
\[
\P\left(\|\wh{\mu}_n - \mu\| \leq \frac{2 \beta v^{1/p}}{\delta^{1/p}n^{(p - 1)/p}}\right) \geq 1 - \delta,
\]
where $\wh{\mu}_n := n^{-1}\sum_{m =1}^n X_m$ denotes the usual sample mean.
\end{lemma}

If the underlying space $\B$ is a Hilbert space and $p = 2$, the multiplicative factor of 2 in the bound below can be dropped. Although the sample mean does concentrate at around the true mean in $2$-smooth Banach spaces, the rate of concentration is provably sub-optimal in terms of its dependence on the failure probability $\delta \in (0, 1)$. In the scalar with $p = 2$, \citet{catoni2012challenging} showed that for any $v$ and $n$, there exists a real-valued distribution with mean zero and $\E|X_n|^2 = v$ such that Lemma~\ref{lem:mean_concentration} is tight up to constants. We generalize this result to $p \in (1, 2]$ in the following proposition, which we prove in Appendix~\ref{sec:proofs}. This lower bound in the scalar setting naturally implies corresponding lower bounds in Banach spaces as well.

\begin{prop}
\label{prop:lower_bd_sample_mean}
Let $p \in (1, 2]$, $v > 0$, $\delta \in (0 , (2e)^{-1}]$, and $n \geq 1$ be arbitrary. Then, there exist real-valued, i.i.d.\ random variables $X_1, \dots, X_n$ such that $\E[X_n] = 0$, $\E|X_n|^p = v$, and 
\begin{align*}
\P\left(\left|\wh{\mu}_n \right| \geq \frac{e^{-1/p}v^{1/p}}{\delta^{1/p}n^{(p - 1)/p}}\right) \geq \delta,
\end{align*}
where $\wh{\mu}_n := \frac{1}{n}\sum_{i = 1}^n X_n$ is the sample mean.
\end{prop}
This sub-optimality motivates the study of other, more intricate estimators, such as geometric median means and our truncation-based estimator. These estimators will obtain exponentially-better dependence on the $\delta$. In particular, the width of the confidence intervals will scale proportionally to $(\log(1/\delta))^{(p - 1)/p}$ instead of $\delta^{-1/p}$.
}

\section{Omitted Proofs}
\label{sec:proofs}

\begin{definition}
\label{def:decouple}

Let $(X_n)_{n \geq 1}$ be a sequence of $\B$-valued random variables on a probability space $(\Omega, \calF, \P)$ and let $(\calF_n)_{n \geq 0}$ denote its natural filtration. We say $(X_n')_{n \geq 1}$ is a decoupled tangent process for $(X_n)_{n \geq 1}$ if the following hold.
\begin{enumerate}
    \item $X_n =_d X_n' \mid \calF_{n - 1}$ for all $n \geq 1$, and
    \item For any $n \geq 1$, $X_1', \dots X_n'$ are conditionally independent given $\calF_\infty := \bigcup_{n = 0}^\infty \calF_n$.
\end{enumerate}

\end{definition}

\begin{prop}[Marcinkiewicz-Zygmund/von Bahr-Esseen in Smooth Banach Spaces]
\label{prop:marcin}
Let $(X_n)_{n \geq 1}$ be random variables in a $\beta$-smooth Banach space such that $\E_{n - 1} X_n = 0$ and $\E\|X_n \|^p < \infty$ for all $n \geq 0$ and some $p \in (1, 2]$. Letting $(S_n)_{n \geq 0}$ be $S_n := X_1 + \cdots + X_n$, it follows that:
\begin{enumerate}
\item If the $(X_n)_{n \geq 1}$ are independent, we have
\[
\E\|S_n\|^p \leq 2^p\beta^p\E\left[\left(\sum_{m = 1}^n \|X_m\|^2\right)^{p/2}\right] \leq  2^p \beta^p \sum_{m=1}^n \E\|X_m\|^p.
\]
\item Otherwise, we just have
\[
\E\|S_n\|^p \leq 2^p\beta^p \sum_{m  = 1}^n \E\|X_m\|^p.
\]
\end{enumerate}
\end{prop}

\begin{proof}

Let us first assume that $(X_n)_{n \geq 1}$ are independent. In such a case, the result follows from combining symmetrization arguments and Jensen's inequality. More specifically, let $(\epsilon_n)_{n \geq 1}$ be a sequence of i.i.d.\ Rademacher random variables that are also independent of $(X_n)_{n \geq 1}$.\footnote{A random variable $\epsilon$ follows a Rademacher distribution if $\P(\epsilon = 1) = \P(\epsilon = - 1) = \frac{1}{2}$.} It suffices to observe that

\begin{align*}
\E\left\|\sum_{m = 1}^n X_m\right\|^p &= \E\left\|\sum_{m = 1}^n (X_n - \E X_n)\right\|^p \\
&\leq \E\left\| \sum_{m = 1}^n (X_m - X_m')\right\|^p &\text{Jensen's inequality} \\
&= \E\left\|\sum_{m = 1}^n \epsilon_m (X_m - X_m') \right\|^p &\text{Symmetry}\\
&\leq 2^p \E\left\|\sum_{m = 1}^n \epsilon_m X_m\right\|^p &\text{Convexity of } \|\cdot\|^p \\
&\leq 2^p \beta^p \E\left(\sum_{m = 1}^n \|X_m\|^2\right)^{p/2},
\end{align*}
to prove the first inequality.
The second inequality follows because $(a + b)^q \leq a^q + b^q$ for  $a, b \geq 0, q \in (0, 1]$.

If $(X_n)_{n \geq 1}$ are not independent, let $(X_n')_{n \geq 1}$ be a decoupled tangent sequence for $(X_n)_{n \geq 1}$ (per Definition~\ref{def:decouple}). 
Let $(\calG_n)_{n \geq 0}$ denote the natural filtration associated with $(X_n')_{n \geq 1}$ and let $(\calH_n)_{n \geq 0}$ denote the join of $\calG$ and $\calF$. Further, let $(\epsilon_n)_{n \geq 0}$ be a sequence of i.i.d.\ Rademacher random variables independent of $\calH_\infty$. First, by construction, we note that $X_n - X_n'$ is conditionally symmetric given $\calF_{n - 1}$. Thus, we have 
\[
\epsilon_n (X_n - X_n') =_{d} X_n - X_n' \mid \calF_{n - 1}.
\]

We prove the desired claim inductively. For $n = 1$, the result is trivially true. Next, observe that we have:
\begin{align*}
&\E\|S_n\|^p = \E\|S_{n - 1} + X_n\|^p \\
&\qquad = \E\|S_{n - 1} + X_n - \E(X_n' \mid \calF_{n - 1})\|^p &\text{(Definition of $X_n'$)} \\
&\qquad \leq \E\left[\E\left(\|S_{n - 1} + X_n - X_n'\|^p \mid \calF_{n - 1}\right)\right] &\text{(Jensen's inequality)} \\
&\qquad = \E\left[\E\left(\|S_{n - 1} + \epsilon_n(X_n - X_n')\|^p \mid \calF_{n - 1}\right)\right] &\text{(Conditional Symmetry)}\\
&\qquad =\E\left[\|S_{n - 1} + \epsilon_n(X_n - X_n')\|^p \right] &\text{(Tower Rule)}\\
&\qquad =\E\left[\E\left(\|S_{n - 1} + \epsilon_n(X_n - X_n')\|^p \mid \calH_n\right)\right] &\text{(Tower Rule)}\\
&\qquad \leq \E\left[\E\left(\|S_{n - 1} + \epsilon_n(X_n - X_n')\|^2\mid \calH_n\right)^{p/2}\right] &\text{(Jensen's inequality)}\\
&\qquad =: (\star).
\end{align*}
Next, define the function $g(\theta) := \E\left(\|S_{n - 1} + \theta \epsilon_n (X_n - X_n')\|^2 \mid \calH_n\right)$. In principle, the norm function need not be differentiable, and the same applies to $g$. Nonetheless, and analogously to the proof of Proposition~\ref{prop:banach}, one can invoke \citet[Remark 2.4]{pinelis1994optimum} to assume smoothness of the norm without loss of generality, alongside the properties exhibited in \citet[Lemma 2.2]{pinelis1992approach}. It follows that
\begin{align*}
g'(0) &= \E\left(\frac{d}{d\theta}\|S_{n - 1} + \theta\epsilon_n(X_n - X_n')\|^2 \mid \calH_n\right) \\
&= 2\E\left(\|S_{n - 1}\|\langle D_v \|v\|\vert_{v = S_{n - 1}}, \epsilon_n(X_n - X_n')\rangle \mid \calH_n\right) \\
&= 2\|S_{n - 1}\|\left\langle D_v\|v\|_{v = S_{ n -1}}, \E[\epsilon_n(X_n - X_n') \mid \calH_n]\right\rangle \\
&= 0.
\end{align*}
Likewise, from the second property exhibited in~\citet[Equation 2.4, Lemma 2.2]{pinelis1992approach}, one has $$g''(\theta) \leq 2\beta^2\E(\|\epsilon_n(X_n - X_n')\|^2 \mid \calH_n).$$
Thus, putting these ideas together, we have
\begin{align*}
&\E(\|S_{n - 1} + \epsilon_n(X_n - X_n')\|^2 \mid \calH_n) = g(1) = g(0) + g'(0) + \int_{0}^1(1 - \theta)g''(\theta)d\theta \\
&\qquad \leq \E\left(\|S_{n - 1}\|^2 \mid \calH_n\right) + \beta^2\E\left(\|\epsilon_n(X_n - X_n')\|^2 \mid \calH_n\right)\\
&\qquad = \|S_{n - 1}\|^2 + \beta^2 \|X_n - X_n'\|^2.
\end{align*}
Plugging this into $(\star)$ above, we obtain
\begin{align*}
(\star) &=  \E\left[\E\left(\|S_{n - 1} + \epsilon_n(X_n - X_n')\|^2\mid \calH_n\right)^{p/2}\right] \\
&\leq \E\left(\|S_{n - 1}\|^2 + \beta^2 \|X_n - X_n'\|^2\right)^{p/2} \\
&\stackrel{(i)}{\leq} \E\|S_{n - 1}\|^{p} + \beta^p\E\|X_n - X_n'\|^{p}  \\
&\leq \E\|S_{n - 1}\|^p + 2^p \beta^p\E\|X_n\|^p \qquad \qquad \text{(Convexity of } \|\cdot\|^p\text{)} \\
&\leq 2^p \beta^p\sum_{m = 1}^n\E\|X_m\|^p \qquad \qquad \text{(Inductive hypothesis)},
\end{align*}
where (i) follows from the fact that $(a + b)^q \leq a^q + b^q$ for all $a, b \geq 0, q \in (0, 1]$. Thus, we have shown the desired result.
\end{proof}

\begin{proof}[\bf Proof of Lemma~\ref{lem:mean_concentration}]
Proposition~\ref{prop:marcin} provides a bound on the $p$-th moment of sums in smooth Banach spaces. The proof of Lemma~\ref{lem:mean_concentration} now follows as a corollary. Applying Markov's/Chebyshev's inequality yields that, for any $t > 0$, we have
\[
\P(\|\wh{\mu}_n - \mu \| \geq t) \leq \frac{\E\|\wh{\mu}_n - \mu\|^p}{t^p} = \frac{\E\left\|\sum_{m = 1}^n (X_m - \mu)\right\|^p}{n^p t^p}.
\]
By Proposition~\ref{prop:marcin}, we see that
\begin{align*}
\E\left\|\sum_{m = 1}^n (X_m - \mu)\right\|^p &\leq 2^p \beta^p \sum_{m = 1}^n \E\|X_m\|^p  \leq n 2^p \beta^p v,
\end{align*}
where the final inequality follows by Assumption~\ref{ass:mean}. 
Thus, we have
\[
\P(\|\wh{\mu}_n - \mu \| \geq t) \leq \frac{n^{1- p}v 2^p\beta^p }{t^p}.
\]
To meet the target failure probability $\delta$, we take $t = \delta^{-\frac{1}{p}}v^{\frac{1}{p}} 2\beta n^{-\frac{p - 1}{p}}$, which proves the desired result.  
\end{proof}

We conclude by proving Proposition~\ref{prop:lower_bd_sample_mean}. Our argument follows by generalizing Proposition 6.2 in \citet{catoni2012challenging}, which concerns the case $p = 2$.

\begin{proof}
Given $p \in (1, 2]$, $v > 0$, $\delta \in (0, (2e)^{-1}]$, and $n \geq 1$, let $X_1, \dots, X_n$ be i.i.d.\ random variables with distribution
\[
\P(X_n = M) = \P(X_n = -M) = \frac{\epsilon}{2} \quad \text{and} \quad \P(X_n = 0) = 1 - \epsilon,
\]
where $\epsilon = \frac{e\delta}{n}$ and $M := \left(\frac{v}{\epsilon}\right)^{1/p}$. Observe that $\E X_n = 0$ and $\E|X_n|^p = v$.

Let $K \sim \Bin(n, \epsilon)$ be a draw from the binomial distribution denoting the number of $X_m$ that are non-zero. We see that on the event $\{K = 1\}$, we have
\[
\left|\wh{\mu}_n\right| = \frac{M}{n} = e^{-1/p}v^{1/p}n^{-(p - 1)/p}\delta^{-1/p}.
\]
Consequently, we have 
\begin{align*}
&\P\left(\left|\wh{\mu}_n\right| \geq e^{-1/p}v^{1/p}n^{-(p - 1)/p}\delta^{-1/p}\right) \geq \P\left(\left|\wh{\mu}_n\right| = e^{-1/p}v^{1/p}n^{-(p - 1)/p}\delta^{-1/p}\right) \\
&\qquad \geq \P(K = 1) \\
&\qquad = n\epsilon(1 - \epsilon)^{n - 1} \\
&\qquad = e \delta\left(1 - \frac{e\delta}{n}\right)^{n - 1} \\
&\qquad \geq e \delta\left(1 - (n - 1)\frac{e\delta}{n}\right) \quad(\text{Bernoulli's Inequality}) \\
&\qquad \geq e\delta\left(1 - \frac{n - 1}{2n}\right) \quad(\text{Since } \delta \leq (2e)^{-1}) \\
&\qquad \geq \delta
\end{align*}
where the final equality follows from the fact that $e\left(1 - \frac{n - 1}{2n}\right) \geq 1$. This completes the proof.
\end{proof}

\section{Geometric Median-of-Means in Banach Spaces}
\label{sec:gmom-banach}
While \citet{minsker2015geometric} studied mean estimation in reflexive Banach spaces, his examples weren't stated explicitly for Banach spaces nor for the case of infinite variance.  
Here we show that his geometric median-of-means estimator, when paired with the sample mean estimation rates outlined in earlier appendices, achieves rate of the form $D_p\beta v^{1/p}(\log(1/\delta)/n)^\frac{p-1}{p}$, where $D_p$ is an absolute constant not depending on $v, \delta$ or, $n$. This is, up to lower order terms, the same as our rate. 

\citet[Theorem 3.1]{minsker2015geometric} provides the following bound.
Let $\widehat{\mu}_1,\dots,\widehat{\mu}_B$ be a collection of independent estimators of the mean $\mu$. Fix $\alpha\in(0,1/2)$. Let $0<\gamma<\alpha$ and $\epsilon>0$ be such that for all $b$, $1\leq b\leq B$, we have 
\begin{equation}
    \P(\| \mu - \widehat{\mu}_b\| >\epsilon)\leq \gamma.
\end{equation}
Let $\widehat{\mu} = \text{median}(\widehat{\mu}_1,\dots,\widehat{\mu}_B)$ be the geometric median, defined as
\[
\widehat{\mu} := \arg\min_{\mu \in \B} \sum_{b=1}^B \|\widehat{\mu}_b - \mu\|.
\]
Then 
\begin{equation}
    \P(\| \mu - \widehat{\mu}\| > C_\alpha \epsilon)\leq \exp(-B \psi(\alpha;\gamma)),
\end{equation}
where 
\[\psi(\alpha;\gamma) = (1-\alpha)\log\frac{1-\alpha}{1-\gamma} + \alpha\log\frac{\alpha}{\gamma},\]
and 
\[C_\alpha = \frac{2(1-\alpha)}{1-2\alpha}.\]
We will optimize Minsker's bound by taking the same optimization parameters as in his paper. That is, we will set $\alpha_\ast := \frac{7}{18}$ and $\gamma_\ast := 0.1$ and will set the number of naive mean estimators $B$ to be given by 
\[
B := \left\lfloor \frac{\log(1/\delta)}{\psi(\alpha_\ast; \gamma_\ast)}\right\rfloor + 1 \leq 3.5\log(1/\delta) + 1,
\]
which provides an overall failure probability of at most 
$\exp(-B\psi(\alpha_\ast; \gamma_\ast)) \leq \delta.$ 
Lemma~\ref{lem:mean_concentration} gives 
\[
\P(\|\wh{\mu}_b - \mu\| \geq \epsilon) \leq \frac{v 2^p\beta^p}{\epsilon^p (n/B)^{p-1}} \leq \gamma_\ast,
\]
if we take  
\[\epsilon = \left( \frac{B^{p-1} v 2^p\beta^p}{\gamma_* n^{p-1}}\right)^{1/p}.\]
Therefore, we obtain that with probability $1-\delta$, 
\begin{align*}
    \|\widehat{\mu} - \mu\| &\leq C_{\alpha_\ast} \epsilon = \frac{4\beta(1-\alpha_\ast)}{1 - 2\alpha_\ast} \left(\frac{B^{p-1} }{\gamma_\ast n^{p-1}}\right)^{1/p} v^{1/p} 
    \\
    &\leq \frac{11\beta}{0.1^{1/p}} v^{1/p} \left(\frac{3.5 \log(1/\delta)+1}{n}\right)^{1-1/p}.
\end{align*}

\section{Noncentral moment bounds}
\label{sec:additional-results}
For completeness, we state our bound when we assume only a bound on the raw (uncentered) $p$-th moment of the observations. This was the setting studied by \citet{catoni2018dimension}. We replace assumption~\ref{ass:mean} with the following: 

\begin{ass}
\label{ass:raw_moment}
We assume $(X_n)_{n \geq 1}$ are a sequence of $\B$-valued random variables adapted to a filtration $(\calF_n)_{n\geq 0}$ such that 
\begin{enumerate}
    \item[(1)] $\E(X_n \mid \calF_{n - 1}) = \mu$, for all $n \geq 1$ and some unknown $\mu \in \B$, and 
    \item[(2)] $\sup_{n \geq 1}\E\left(\|X_n \|^p \mid \calF_{n - 1}\right) \leq v < \infty$ for some known $p \in (1, 2]$ and some known constant $v > 0$.
\end{enumerate}
\end{ass}

With only the raw moment assumption, we do not try and center our estimator. Instead we deploy $\est_n(0, \lambda,0) = \frac{1}{n}\sum_{m\leq n} \Trunc(\lambda X_m)X_m$. With this estimator we obtain the following result, which achieves the same rate as \citet{catoni2018dimension} and \citet{chugg2023time}. 

\begin{theorem}
 Let $X_1,X_2,\dots$ be random variables satisfying Assumption~\ref{ass:raw_moment} which live in some Banach space $(\B,\|\cdot\|)$ satisfying Assumption~\ref{ass:smooth}. 
Fix any $\delta \in (0, 1]$. Then, for any $\lambda > 0$, with probability $1-\delta$, simultaneously for all $n\geq 1$, we have: 
\begin{equation}
\left\| \est_{n}(0,\lambda,0) - \mu\right\| \leq 2v\lambda^{p-1}(\beta\const_p(\B) + K_p2^{p-1}) + \frac{\beta\log(2/\delta)}{\lambda n}. 
\end{equation}
Moreover, if we want to optimize the bound at a particular sample size $n$ and we set $\rho = \beta^{-1}$
\[\lambda = \left(\frac{\log(2/\delta)\beta}{2nv (\beta\const_p(\B) + K_p2^{p-1})}\right)^{1/p},\]
then 
\begin{equation}
\label{eq:bound-uncentered}
\P\left(\left\| \est_{n}(0,\lambda,0) - \mu\right\| \geq 2^{(p + 1)/p}v^{1/p}\beta C_p^{1/p}\left(\frac{\log(2/\delta)}{n}\right)^{(p - 1)/p}\right) \leq \delta,
\end{equation}
where $C_p$ is as stated in Corollary~\ref{cor:opt-lambda}.
\end{theorem}
\begin{proof}
Apply Theorem~\ref{thm:template} with $k=0$ and $\zhat_k=0$. Then note that we can take $r(\delta, 0) = v^{1/p}$ for all $\delta$ since $\| \mu\| \leq (\E \|X\|^p)^{1/p}\leq v^{1/p}$ by Jensen's inequality. 
\end{proof}

\edit{\section{Further experiments} \label{sec:further_experiments}

We extend the experiments exhibited in Section~\ref{sec:experiments} for further choices of the hyperparameter~$\lambda$. In particular, we consider $\lambda \in [0.0005, 0.005, 0.05, 0.5]$ for both the sampling mechanisms presented in Section~\ref{sec:experiments}, where $\lambda$ is the truncation level as optimized in Corollary~\ref{cor:opt-lambda}. Figure~\ref{fig:est_dist} and Figure~\ref{fig:zeta5} display the results corresponding to the finite and infinite dimensional spaces, respectively. We only include the sample mean in the first plot as to not compress the empirical distributions associated with other estimates. In both cases, we see that the geometric median-of-means estimate is roughly two times closer to the mean than either truncation-based estimate for $\lambda \in \{0.0005, 0.005\}$ (corresponding to truncation at large radii). In the setting of aggressive truncation ($\lambda \in \{0.05, 0.5\}$), the truncated mean estimator centered at the geometric median-of-means initial estimate offers a significantly smaller distance to the true mean than just geometric median-of-means alone. The truncated estimate centered at the sample mean performs similarly for $\lambda = 0.05$, but suffers heavily from outliers when $\lambda = 0.5$. Interestingly, the recommended truncation level for optimizing tightness at $n = 1e5$ samples is $\lambda \approx 0.0004$ per Corollary~\ref{cor:opt-lambda}. Our experiments reflect that one may want to truncate more aggressively than is recommended in the corollary.  In practice, one could likely choose an appropriate truncation level through cross-validation.

}

\begin{figure}
    \centering
    \begin{subfigure}{0.4\textwidth}
    \centering
    \includegraphics[width=\linewidth]{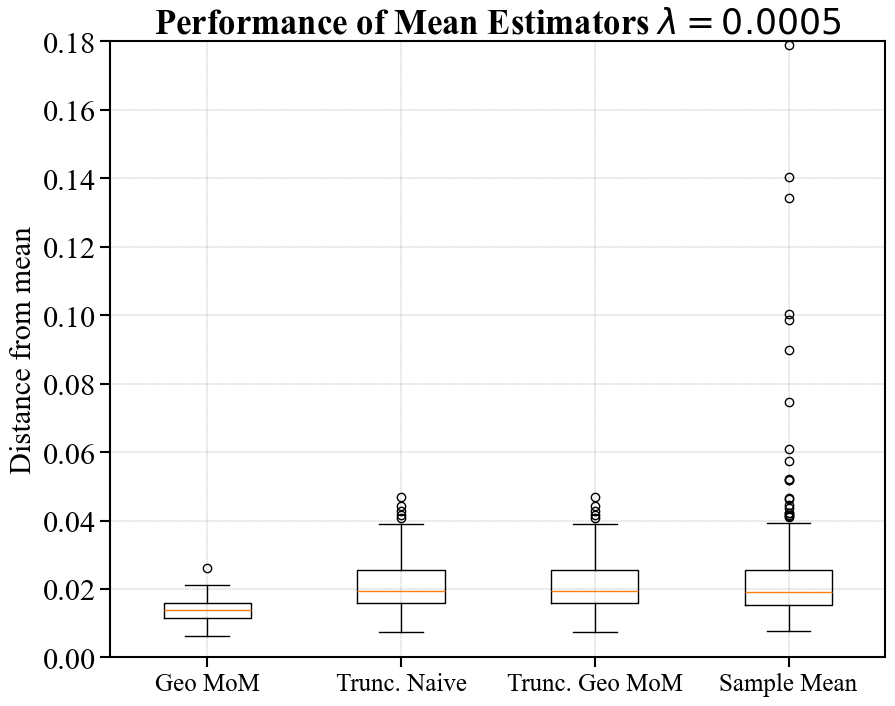}
    \caption{$\lambda = 0.0005$}
    \label{fig:est_dist:lam=5e-4}
    \end{subfigure}
    \begin{subfigure}{0.4\textwidth}
    \centering
    \includegraphics[width=\linewidth]{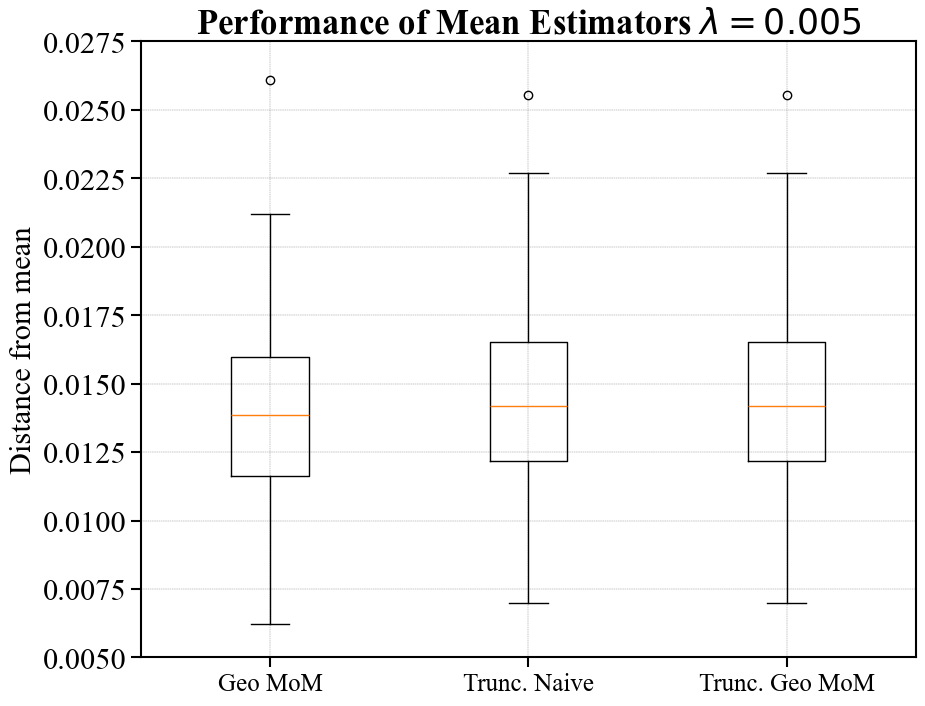}
    \caption{$\lambda = 0.005$}
    \end{subfigure}
    \begin{subfigure}{0.4\textwidth}
        \centering
    \includegraphics[width=\linewidth]{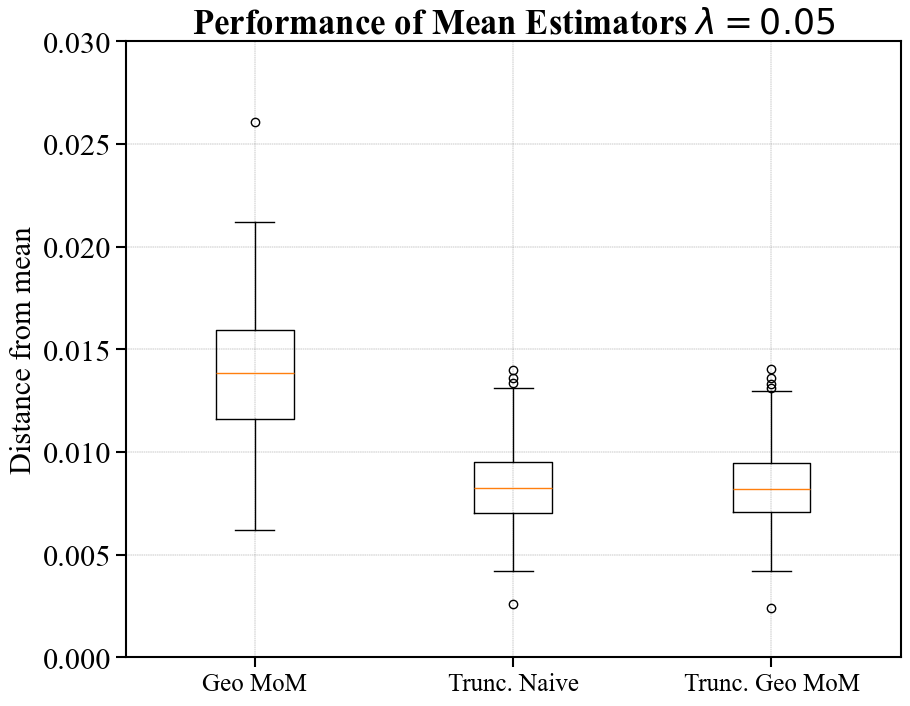}
    \caption{$\lambda = 0.05$}
    \label{fig:est_dist:lam=5e-2}
    \end{subfigure}
        \begin{subfigure}{0.4\textwidth}
        \centering
    \includegraphics[width=\linewidth]{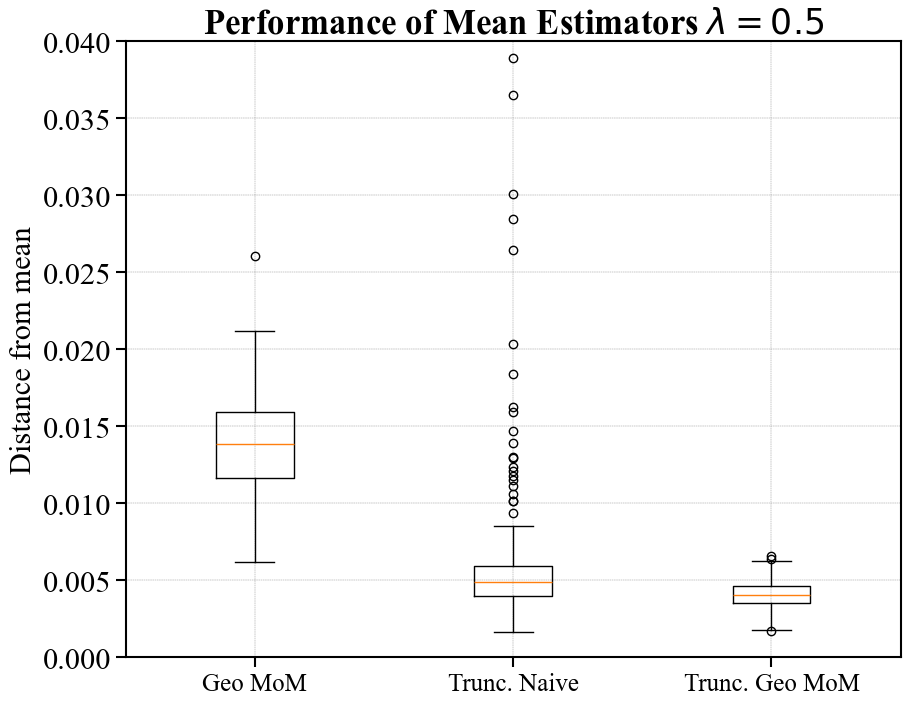}
    \caption{$\lambda = 0.5$}
    \label{fig:est_dist:lam=5e-1}
    \end{subfigure}
    \caption{We compare the empirical distributions of distance between the mean estimate and the true mean for a variety of estimators in the finite dimensional space $\R^{10}$. We consider the sample mean, the geometric median-of-means~\citep{minsker2015geometric}, and our truncation-based estimators from Theorem~\ref{thm:template} (with either the sample mean or geometric median-of-means for the initial mean estimates). We generate $n =10^6$ i.i.d.\ samples and use $k = \lfloor\sqrt{10^6}\rfloor$ samples to construct initial mean estimates. We compute these estimates on 250 runs. For our truncation-based estimates, we consider $\lambda \in [0.0005, 0.005, 0.05, 0.5]$, where $\lambda$ is the truncation level optimized as in Corollary~\ref{cor:opt-lambda}. We only include the sample mean in the first plot for readability. }
    \label{fig:est_dist}
\end{figure}

\begin{figure}
    \centering
    \begin{subfigure}{0.4\textwidth}
    \centering
    \includegraphics[width=\linewidth]{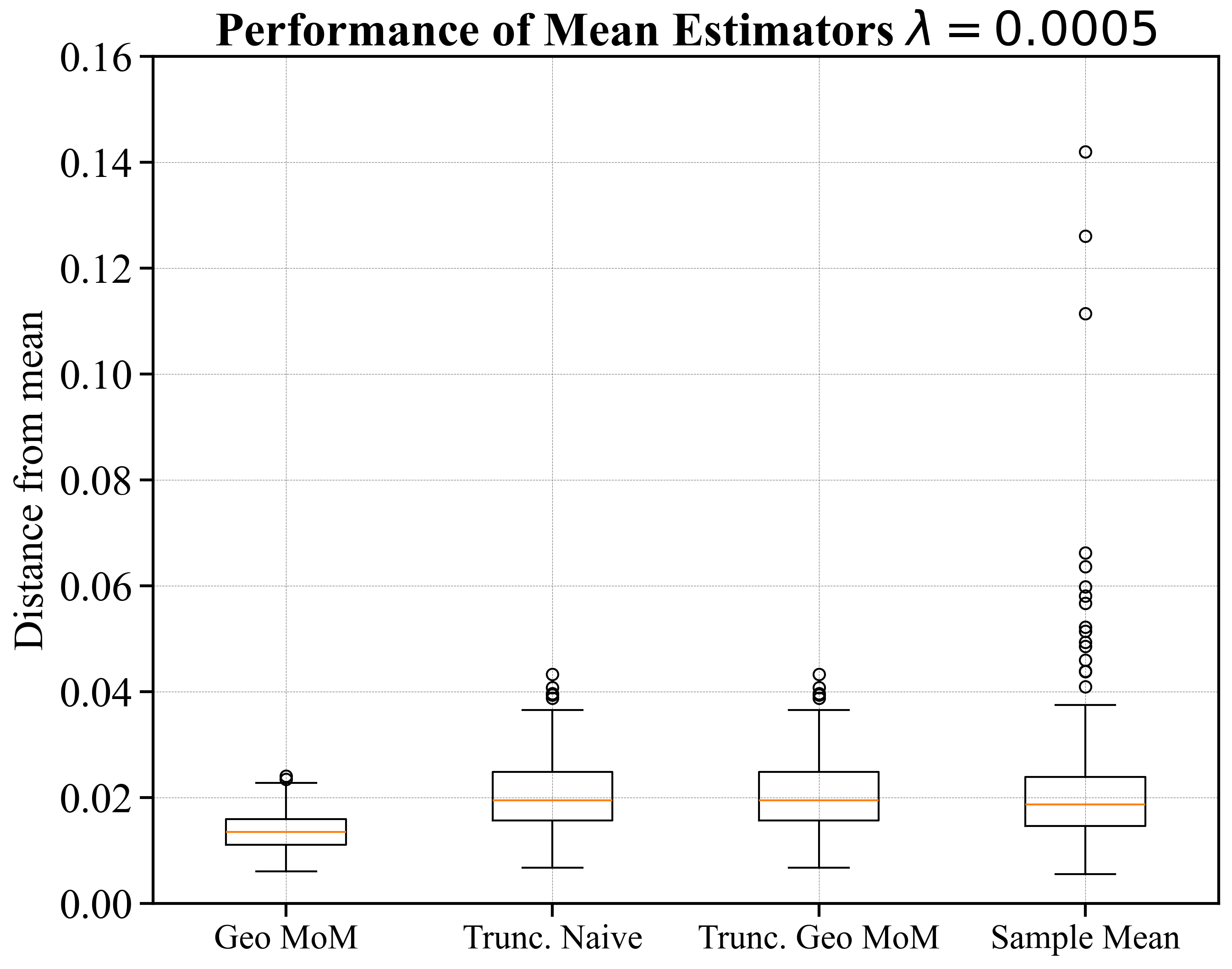}
    \caption{$\lambda = 0.0005$}
    \label{fig:lam0005_zeta5}
    \end{subfigure}
    \begin{subfigure}{0.4\textwidth}
    \centering
    \includegraphics[width=\linewidth]{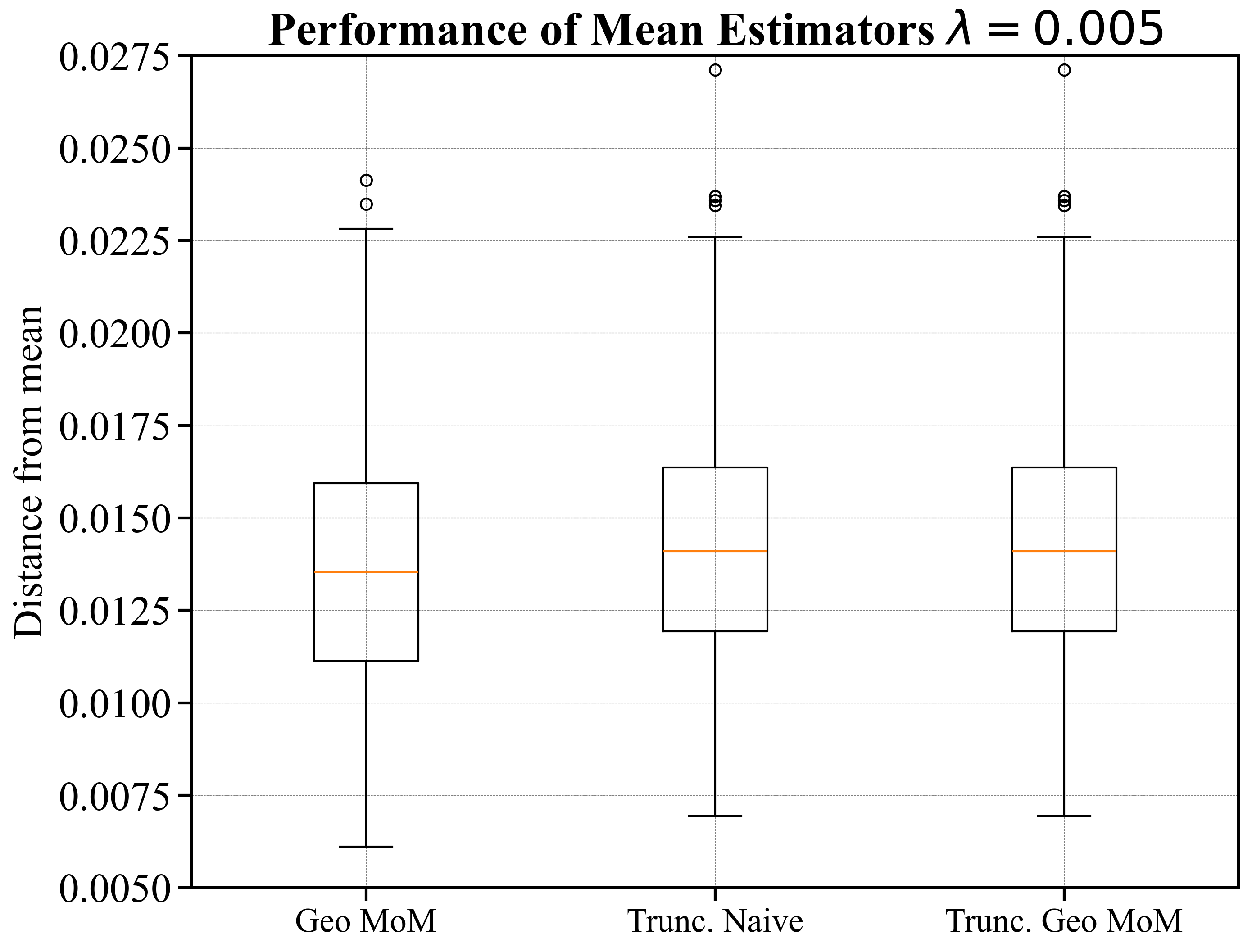}
    \caption{$\lambda = 0.005$}
    \label{fig:lam005_zeta5}
    \end{subfigure}
    \begin{subfigure}{0.4\textwidth}
        \centering
    \includegraphics[width=\linewidth]{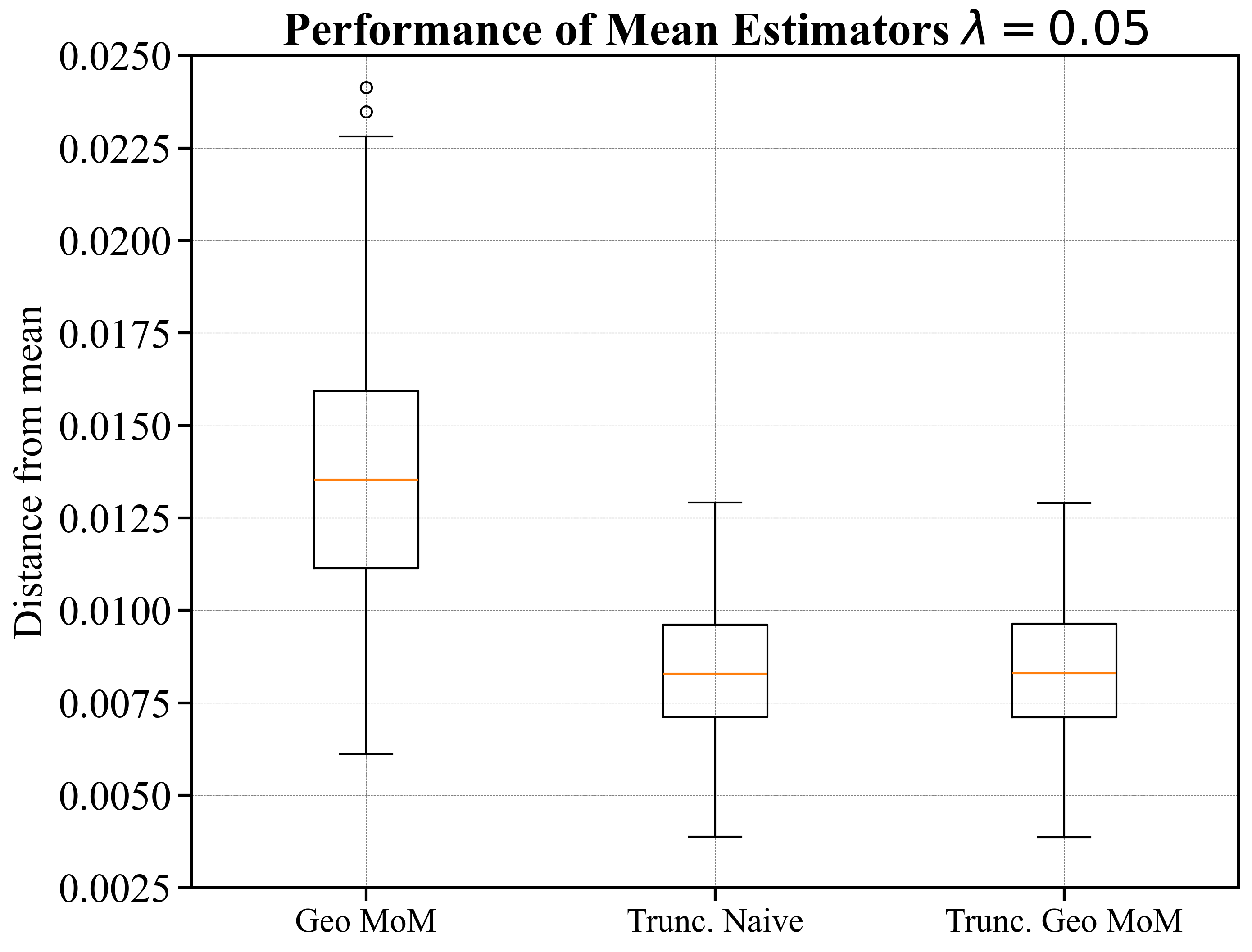}
    \caption{$\lambda = 0.05$}
    \label{fig:lam05_zeta5}
    \end{subfigure}
        \begin{subfigure}{0.4\textwidth}
        \centering
    \includegraphics[width=\linewidth]{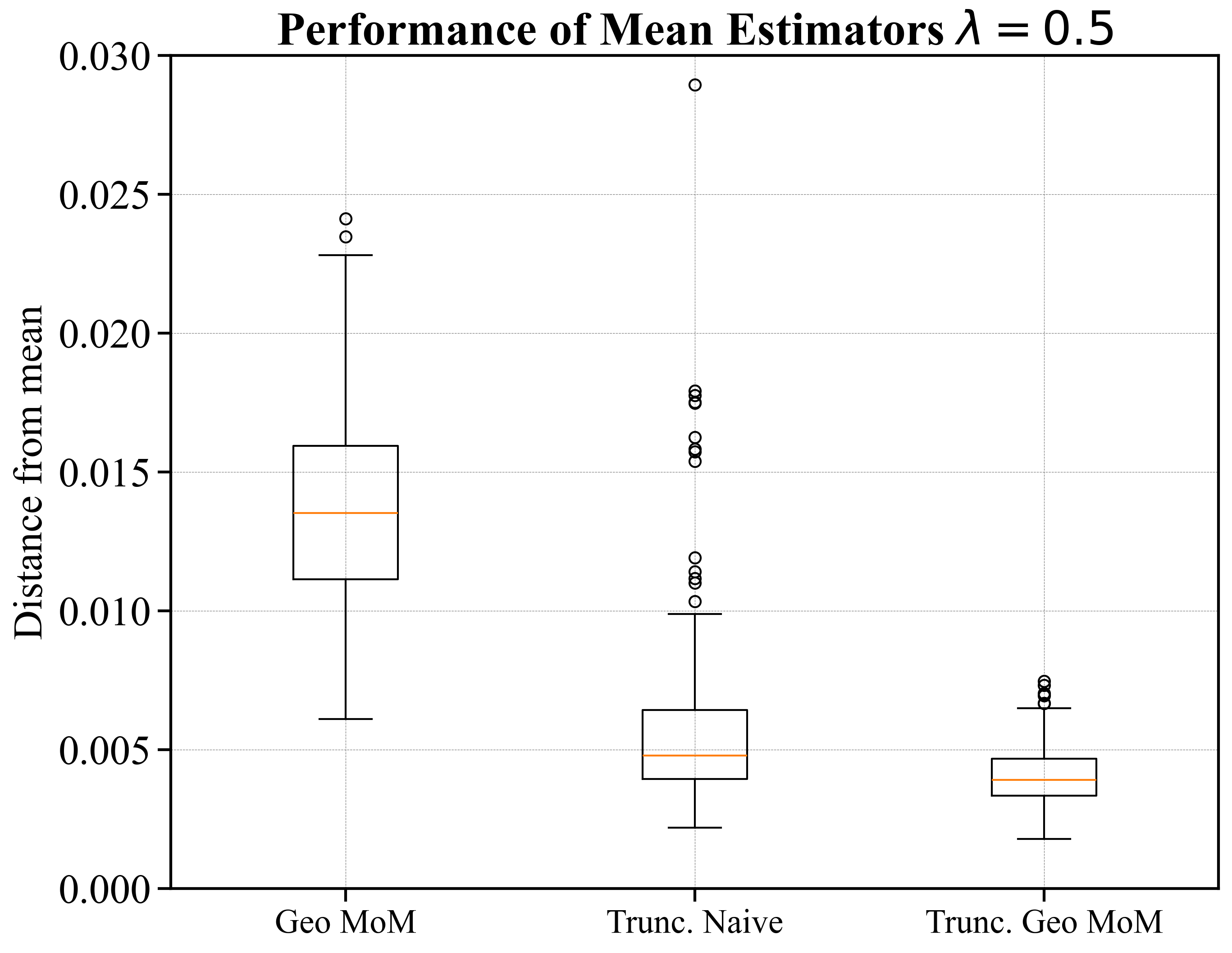}
    \caption{$\lambda = 0.5$}
    \label{fig:lam5_zeta5}
    \end{subfigure}
    \caption{We compare the empirical distributions of distance between the mean estimate and the true mean for a variety of estimators in an infinite dimensional sequence space  $\ell^2(\N)$. 
    Other experimental details are identical to Figure~\ref{fig:est_dist}. 
    }
    \label{fig:zeta5}
\end{figure}

\end{document}